\documentclass[10pt]{amsart}

\usepackage{enumerate}
\usepackage{comment}

\makeatletter

\@namedef{subjclassname@2010}{

  \textup{2010} Mathematics Subject Classification}

\usepackage{amssymb, amsmath,amsthm}
\usepackage{mathrsfs}
\newtheorem{thm}{Theorem}[section]
\newtheorem*{thm1}{Theorem}
\newtheorem{prop}[thm]{Proposition}

\newtheorem{cor}[thm]{Corollary}
\newtheorem{lem}[thm]{Lemma}
\theoremstyle{definition}
\newtheorem{rem}[thm]{Remark}

\newcommand{\ra}{\rightarrow}
\newcommand{\bk}{\backslash}
\newcommand{\mc}{\mathcal}
\newcommand{\mf}{\mathfrak}
\newcommand{\mb}{\mathbb}
\newcommand{\sg}{\sigma}

\renewcommand{\ss}{\substack}
\newcommand{\llf}{\left\lfloor}

\newcommand{\e}{\varepsilon}
\newcommand{\rrf}{\right\rfloor}

\renewcommand{\bar}{\overline}

\begin{document}
\title{Divisor-Bounded Multiplicative Functions in Short Intervals}
\author{Alexander P. Mangerel}
\address{Centre de Recherches Math\'{e}matiques, Universit\'{e} de Montr\'{e}al, Montr\'{e}al, Qu\'{e}bec}
\email{smangerel@gmail.com}

\begin{abstract}
We extend the Matom\"{a}ki-Radziwi\l\l{} theorem to a large collection of unbounded multiplicative functions that are uniformly bounded, but not necessarily bounded by 1, on the primes. Our result allows us to estimate averages of such a function $f$ in typical intervals of length $h(\log X)^c$, with $h = h(X) \ra \infty$ and where $c = c_f \geq 0$ is determined by the distribution of $\{|f(p)|\}_p$ in an explicit way. We give three applications.\\
First, we show that the classical Rankin-Selberg-type asymptotic formula for partial sums of $|\lambda_f(n)|^2$, where $\{\lambda_f(n)\}_n$ is the sequence of normalized Fourier coefficients of a primitive non-CM holomorphic cusp form, persists in typical short intervals of length $h\log X$, if $h = h(X) \ra \infty$. We also generalize this result to sequences $\{|\lambda_{\pi}(n)|^2\}_n$, where $\lambda_{\pi}(n)$ is the $n$th coefficient of the standard $L$-function of an automorphic representation $\pi$ with unitary central character for $GL_m$, $m \geq 2$, provided $\pi$ satisfies the generalized Ramanujan conjecture. \\
Second, using recent developments in the theory of automorphic forms we estimate the variance of averages of all positive real moments $\{|\lambda_f(n)|^{\alpha}\}_n$ over intervals of length $h(\log X)^{c_{\alpha}}$, with $c_{\alpha} > 0$ explicit, for any $\alpha > 0$, as $h = h(X) \ra \infty$.\\
Finally, we show that the (non-multiplicative) Hooley $\Delta$-function has average value $\gg \log\log X$ in typical short intervals of length $(\log X)^{1/2+\eta}$, where $\eta >0$ is fixed.
\end{abstract}

\maketitle

\section{Introduction and Main Results} \label{sec:MRDB}
\subsection{The Matom\"{a}ki-Radziwi\l\l{} theorem for bounded multiplicative functions}
The Matom\"{a}ki-Radziwi\l\l{} theorem, in its various incarnations, gives estimates for the error term in approximating the average of a bounded multiplicative function in a typical short interval by a corresponding long interval average. In the breakthrough paper \cite{MR}, the authors showed that, uniformly over all real-valued multiplicative functions $f: \mb{N} \ra [-1,1]$, for any $1 \leq h \leq X$ such that $h = h(X) \ra \infty$ as $X \ra \infty$,
\begin{equation}\label{eq:MRqual}
\frac{1}{h}\sum_{x-h < n \leq x} f(n) = \frac{2}{X} \sum_{X/2 < n \leq X} f(n) + o(1)
\end{equation}
for all but $o(X)$ integers $x \in [X/2,X]$. A key feature of this result is that the interval length $h$ can grow arbitrarily slowly as a function of $X$. This result has had countless applications to a variety of problems across mathematics, including to partial results towards Chowla's conjecture on correlations of the Liouville function \cite{Tao}, \cite{TaoTer}, the resolution of the famous Erd\H{o}s discrepancy problem \cite{EDP}, and progress on Sarnak's M\"{o}bius disjointness conjecture (e.g., \cite{TaoSar}, \cite{FranHost}; see \cite{SarSur} for a more exhaustive list).

Since \cite{MR}, the result has been extended and generalized in various directions. In \cite{MRT}, a corresponding short interval result was given for non-pretentious complex-valued multiplicative functions $f:\mb{N} \ra \mb{U}$, where $\mb{U} := \{z \in \mb{C} : |z| \leq 1\}$. To be precise, if we define
$$
D_f(X;T) := \min_{|t| \leq T} \mb{D}(f,n^{it};X)^2 := \min_{|t| \leq T}\sum_{p \leq X} \frac{1-\text{Re}(f(p)p^{-it})}{p},
$$
where $\mb{D}$ denotes the Granville-Soundararajan pretentious distance, they showed that if $f: \mb{N} \ra \mb{U}$ satisfies $D_f(X;X) \ra \infty$ then
$$
\left|\frac{1}{h} \sum_{x-h < n \leq x} f(n)\right| = o(1)
$$
for all but $o(X)$ integers $x \in [X/2,X]$, whenever $h = h(X) \ra \infty$. In a different direction, exploring the heuristic relationship between the distributions of arithmetic functions in short intervals and in short arithmetic progressions, Klurman, the author and Ter\"{a}v\"{a}inen \cite{KMT} obtained an analogue of \eqref{eq:MRqual} for typical\footnote{Complications arise concerning both the prime divisors of the modulus $q$ as well as the distribution of zeros of Dirichlet $L$-functions $\pmod{q}$, so the theorem proven in \cite{KMT} is qualitatively weaker than \eqref{eq:MRqual} unconditionally in general.} short arithmetic progressions.

In the recent paper \cite{MRII}, a widely generalized version of the results of \cite{MR} was developed, which among other things extended the work of \cite{MRT}. The authors showed that for a general complex-valued multiplicative function $f: \mb{N} \ra \mb{U}$, if $t_0 = t_0(f,X)$ is a minimizer in the definition of $D_f(X;X)$ then for all but $o(X)$ integers $x \in [X/2,X]$, one obtains an asymptotic formula with main term of the form
\begin{equation}\label{eq:MRqualCV}
\frac{1}{h}\sum_{x-h < n \leq x} f(n) = \frac{1}{h}\int_{x-h}^x u^{it_0} du \cdot \frac{2}{X}\sum_{X/2 < n \leq X} f(n)n^{-it_0} + o(1),
\end{equation}
with a better quantitative dependence of the bound for the exceptional set on the interval length $h$ than in \cite{MR}. 

By Shiu's theorem (Lemma \ref{lem:Shiu} below), we have
\[
\frac{1}{X}\sum_{n \leq X} |f(n)| \ll \prod_{p \leq X} \left(1+ \frac{|f(p)|-1}{p}\right),
\]
so \eqref{eq:MRqualCV} is trivial whenever $\sum_{p \leq X} \frac{1-|f(p)|}{p} \ra \infty$, for instance if $f(p) = 0$ significantly often on the primes.
Rectifying this weakness, Matom\"{a}ki and Radziwi\l\l{} improved the quality of the $o(1)$ error term for a large collection of $1$-bounded functions with \emph{sparse} prime support. Specifically, they showed that if there are constants $A > 0$ and $\theta \in (0,1]$ such that the sieve-type lower bound condition
\begin{equation}\label{eq:sieveforF}
\sum_{z < p \leq w} \frac{|f(p)|}{p} \geq A\sum_{z <  p \leq w} \frac{1}{p} - O\left(\frac{1}{\log z}\right) \text{ holds for all } 2 \leq z \leq w \leq X^{\theta}
\end{equation}
then one can improve the $o(1)$ term to 
$$
o\left( \prod_{p \leq X} \left(1+\frac{|f(p)|-1}{p}\right)\right).
$$
This savings comes at a natural cost, namely that the length of the interval $h$ is no longer arbitrarily slow growing as a function of $X$, but must grow in a manner that depends on the sparseness of the support\footnote{As pointed out in \cite[p. 8]{MRII}, it is generally unclear what the least size of such intervals must be for a given bounded multiplicative function.} of $f$. 

Precisely, the main result of \cite{MRII} may be stated as follows. In the sequel, for a multiplicative function $f: \mb{N} \ra \mb{U}$ we write
$$
H(f;X) := \prod_{p \leq X} \left(1+\frac{(|f(p)|-1)^2}{p}\right).
$$
\begin{thm1}[Matom\"{a}ki-Radziwi\l\l{}, \cite{MRII} Thm. 1.9]
Let $A > 0$ and $\theta \in (0,1]$. Let $f: \mb{N} \ra \mb{U}$ be a multiplicative function such that \eqref{eq:sieveforF} holds for all $2 \leq z\leq w \leq X^{\theta}$. Let $2 \leq h_0 \leq X^{\theta}$ and put $h := h_0H(f;X)$. Also, set $t_0 = t_0(f,X)$. Then there are constants\footnote{In \cite{MRII} they obtained the explicit constant $\rho_A =  A/3 - \frac{2}{3\pi} \sin(\pi A/2)$.} $C = C(\theta) > 1$, $\rho_A > 0$ such that for any $\delta \in (0,1/1000)$ and $0 < \rho < \rho_A$,
\begin{align*}
&\left|\frac{1}{h} \sum_{x-h < n \leq x} f(n) - \frac{1}{h}\int_{x-h}^x u^{it_0} du \cdot \frac{2}{X}\sum_{X/2 < n\leq X} f(n)n^{-it_0}\right| \\
&\leq \left(\delta + C\left(\frac{\log\log h_0}{\log h_0}\right)^A + (\log X)^{-A\rho/36}\right) \prod_{p \leq X} \left(1+\frac{|f(p)|-1}{p}\right),
\end{align*}
for all $x \in [X/2,X]$ outside of a set of size
$$
\ll_{\theta} X\left(h^{-(\delta/2000)^{1/A}} + X^{-\theta^3(\delta/2000)^{6/A}}\right).
$$
\end{thm1}

\subsection{Divisor-bounded multiplicative functions} \label{sec:DBHist}
Let $B \geq 1$. We define the \emph{generalized $B$-divisor function} $d_B(n)$ via
$$
\zeta(s)^B = \sum_{n \geq 1} \frac{d_B(n)}{n^s} \text{ for } \text{Re}(s) > 1.
$$
It can be deduced that $d_B(n)$ is multiplicative, and moreover $d_B(p^k) = \binom{B+k-1}{k}$, for all $k \geq 1$. In particular, $d_B(p) = B$. For integer values of $B$ this coincides with the usual $B$-fold divisor functions, e.g., when $B = 2$ we have $d_B(n) = d(n)$, and when $B = 1$ we have $d_B(n) \equiv 1$. 

We say that a multiplicative function $f: \mb{N} \ra \mb{C}$ is \emph{divisor-bounded} if there is a $B \geq 1$ such that $|f(n)| \leq d_B(n)$ for all $n$. When $B = 2$, for example, this includes functions such as the twisted divisor function $d(n,\theta) := \sum_{d|n} d^{i\theta}$ for $\theta \in \mb{R}$, as well as $r(n)/4$, where $r(n) := |\{(a,b) \in \mb{Z}^2 : a^2+b^2 =n\}|$. 

There is a rich literature about mean values of general, 1-bounded multiplicative functions. The works of Wirsing \cite{WirMV} and Hal\'{a}sz \cite{Hal} are fundamental, with noteworthy developments by Montgomery \cite{MonMV} and Tenenbaum \cite[Thm. III.4.7]{Ten}. The theory has recently undergone an important change in perspective, due in large part to the extensive, pioneering works of Granville and Soundararajan (e.g., \cite{GSDec}, \cite{GSPret}). This well-formed theory significantly informs the results of \cite{MR} and \cite{MRII}. 

In comparison, the study of long averages of general \emph{unbounded} multiplicative functions has only garnered significant interest more recently. Granville, Harper and Soundararajan \cite{GHS}, in developing a new proof of a quantitative form of Hal\'{a}sz' theorem, were able to obtain bounds for averages of multiplicative functions $f: \mb{N} \ra \mb{C}$ for which the coefficients of the Dirichlet series\footnote{Implicitly, it is assumed that $-\frac{L'}{L}(s,f)$ is well-defined in $\text{Re}(s) > 1$.} 
\begin{equation}\label{eq:LambdafDef}
-\frac{L'}{L}(s,f) = \sum_{n \geq 1} \frac{\Lambda_f(n)}{n^s}, \text{ where } L(s,f) := \sum_{n \geq 1} \frac{f(n)}{n^s} \text{ for } \text{Re}(s) > 1, 
\end{equation}
satisfy the bound $|\Lambda_f(n)| \leq B \Lambda(n)$ uniformly over $n \in \mb{N}$ for some $B \geq 1$, where $\Lambda(n)$ is the von Mangoldt function. Such functions satisfy $|f(n)| \leq d_B(n)$ for all $n$.  In \cite{TenVM}, Tenenbaum, improving on qualitative results due to Elliott \cite{EllMV}, established quantitative upper bounds and asymptotic formulae for the ratios $|\sum_{n \leq X} f(n)|/(\sum_{n \leq X} |f(n)|)$, assuming $f$ is uniformly bounded on the primes, not too large on average at prime powers, and satisfies a hypothesis like \eqref{eq:sieveforF}. See also \cite[Ch. 2]{ManThe} for results of a similar kind under stronger hypotheses. 

On the basis of these developments, it is reasonable to ask whether the results of \cite{MR} and \cite{MRII} can be extended to divisor-bounded functions of a certain type. This was hinted at in \cite[p. 9]{MRII} but, as far as the author is aware, it does not yet exist in the literature. The purpose of this paper is to establish such extensions for a broad class of divisor-bounded multiplicative functions, among other unbounded functions. 

In the following subsection we provide three examples that motivate our main theorem, Theorem \ref{thm:MRFull}. Besides the applications we give here, this result will also be applied in \cite{ManErd} to study short interval averages of general additive functions. 

\subsection{Applications} \label{subsec:apps}

\subsubsection{Rankin-Selberg estimates for $GL_m$ in typical short intervals}
Let $f$ be a fixed even weight $k \geq 2$, level $1$ primitive, Hecke-normalized holomorphic cusp form without complex multiplication, and write its Fourier expansion at $\infty$ as
$$
f(z) = \sum_{n \geq 1} \lambda_f(n)n^{\frac{k-1}{2}}e(nz), \quad \text{Im}(z) > 0,
$$
with $\lambda_f(1) = 1$. Set also
$$
g_f(n) := \sum_{d^2|n} |\lambda_f(n/d^2)|^2.
$$
By the Hecke relation
$$
\lambda_f(m)\lambda_f(n) = \sum_{d|(m,n)} \lambda_f\left(\frac{mn}{d^2}\right), \quad m,n \in \mb{N},
$$
$|\lambda_f|^2$ and thus also $g_f$ are multiplicative functions. 
Deligne showed that $|\lambda_f(p)| \leq 2$ for all primes $p$, and in general $|\lambda_f(n)|^2 \leq d(n)^2$. Thus $|\lambda_f|^2$ is bounded by a power of a divisor-function, and the same can be shown for $g_f$.

The classical Rankin-Selberg method shows \cite[Sec. 14.9]{IK} that asymptotic formulae
\begin{align*}
\frac{1}{X}\sum_{n \leq X} |\lambda_f(n)|^2 = c_f + O(X^{-2/5}), \quad \quad 
\frac{1}{X}\sum_{n \leq X} g_f(n) = d_f + O(X^{-2/5}),
\end{align*}
hold as $X \ra \infty$, where $c_f,d_f > 0$ are constants depending on $f$. The Rankin-Selberg problem is equivalent to asking for an improvement of the error term $X^{-2/5}$ in both of these estimates, but this is not our point of interest here.

One can ask whether the above asymptotic formulae continue to hold in short intervals.   Ivi\'{c} \cite{Ivic} considered
the variance of the error term in short interval averages. Specifically, he showed \cite[Cor. 2]{Ivic} on the Lindel\"{o}f hypothesis that
\begin{equation}\label{eq:Ivic}
\frac{1}{X} \int_X^{2X} \left(\frac{1}{h} \sum_{x < n \leq x+h} g_f(n) - d_f\right)^2 dx = o(1),
\end{equation}
as long as $h \geq X^{2/5-\e}$, albeit with a power-saving error term in the latter range. \\
At the expense of the quality of the error term, we obtain the following improvement in the range where \eqref{eq:Ivic} holds.
\begin{cor}\label{cor:RankinSelberg}
Let $10 \leq h_0 \leq X/(10\log X)$ and set $h := h_0 \log X$. Then there is a constant $\theta > 0$ such that
$$
\frac{1}{X}\int_X^{2X} \left(\frac{1}{h} \sum_{x < n \leq x+h} |\lambda_f(n)|^2 - c_f\right)^2 dx \ll \frac{\log\log h_0}{\log h_0} + \frac{\log\log X}{(\log X)^{\theta}}.
$$
The same estimate holds when $|\lambda_f|^2$ and $c_f$ are replaced by $g_f$ and $d_f$, respectively.
\end{cor}
This corollary might appear surprising, given our currently incomplete understanding of the shifted convolution problem
$$
\sum_{X < n \leq 2X} |\lambda_f(n)|^2|\lambda_f(n+r)|^2, \quad\quad 1 \leq |r| \leq h.
$$
Actually, our proof of Corollary \ref{cor:RankinSelberg} relies only on the multiplicativity of $|\lambda_f|^2$, Deligne's theorem and the prime number theorem for Rankin-Selberg $L$-functions (see e.g., Lemma \ref{lem:PNTRS}). This suggests\footnote{We would like to thank Maksym Radziwi\l\l{} and Jesse Thorner for pointing this out.} that a generalization to coefficients of automorphic $L$-functions for $GL_n$ should be possible, provided that these satisfy the generalized Ramanujan conjecture and hence are divisor-bounded. \\
To be more precise, let $m \geq 2$, let $\mb{A}$ be the ring of adeles of $\mb{Q}$, and let $\pi$ be a cuspidal automorphic representation of $GL_m(\mb{A})$ with unitary central character that acts trivially on the diagonally embedded copy of $\mb{R}^+$. We let $q_{\pi}$ denote the conductor of $\pi$.
The finite part of $\pi$ factors as a tensor product $\pi = \otimes_p \pi_p$, with local representations $\pi_p$ at each prime $p$. The local $L$-function at $p$ takes the form
$$
L(s,\pi_p) = \prod_{1 \leq j \leq m} \left(1-\frac{\alpha_{j,\pi}(p)}{p^s}\right)^{-1} := \sum_{l \geq 0} \frac{\lambda_{\pi}(p^l)}{p^{ls}},
$$
where $\{\alpha_{1,\pi}(p),\ldots,\alpha_{m,\pi}(p)\} \subset \mb{C}$ are the Satake parameters of $\pi_p$. The standard $L$-function of $\pi$ is then
$$
L(s,\pi) := \prod_p L(s,\pi_p) = \sum_{n \geq 1} \frac{\lambda_{\pi}(n)}{n^s},
$$
which converges absolutely when $\text{Re}(s) > 1$.
The sequence of coefficients $\lambda_{\pi}(n)$ thus defined is multiplicative, with the property that
$$
\lambda_{\pi}(p^r) = \sum_{\ss{r_1,\ldots,r_m \geq 0 \\ r_1 + \cdots + r_m = r}} \prod_{1 \leq j \leq m} \alpha_{j,\pi}(p)^{r_j}.
$$
The generalized Ramanujan conjecture (GRC) implies that for all $1 \leq j \leq m$, $|\alpha_{j,\pi}(p)|= 1$  whenever $p \nmid q_{\pi}$ and otherwise $|\alpha_{j,\pi}(p)| \leq 1$. It follows that if $\pi$ satisfies GRC then
$$
|\lambda_{\pi}(p^r)| \leq \sum_{\ss{r_1,\ldots,r_m \geq 0 \\ r_1 + \cdots + r_m = r}} 1 = \binom{m+r-1}{r} = d_m(p^r),
$$
and therefore that $|\lambda_{\pi}(n)| \leq d_m(n)$. 
As a consequence of these properties we will prove the following.
\begin{thm}\label{thm:genAutForms}
Let $m \geq 2$ and let $\pi$ be a fixed cuspidal automorphic representation for $GL_m(\mb{A})$ as above. Assume that $\pi$ satisfies GRC. Let $10 \leq h_0 \leq X/(10(\log X)^{m^2-1})$ and let $h:= h_0 (\log X)^{m^2-1}$. Then there is a constant $\theta = \theta(m) > 0$ such that
$$
\frac{1}{X}\int_X^{2X} \left(\frac{1}{h}\sum_{x< n \leq x + h} |\lambda_{\pi}(n)|^2 - \frac{1}{X}\sum_{X < n \leq 2X} |\lambda_{\pi}(n)|^2 \right)^2 dx \ll \frac{\log\log h_0}{\log h_0} + \frac{\log\log X}{(\log X)^{\theta}}.
$$
\end{thm}
\begin{rem}
In the case $m =2$ the parameter $h$ must grow faster than $(\log X)^3$ in Theorem \ref{thm:genAutForms}, whereas Corollary \ref{cor:RankinSelberg} allows any $h$ growing faster than $\log X$. This is due to the fact that the range of $h$ in these estimates depends on the size of $\sum_{p \leq X} |\lambda_{\pi}(p)|^4/p$. When $\pi = \pi_f$ for a cusp form $f$ on $GL_2(\mb{A})$ we may estimate this sum using the well-known expression
$$
|\lambda_f(p)|^4 = 2 + 3\lambda_{\text{Sym}^2 f}(p) + \lambda_{\text{Sym}^4 f}(p)
$$
for all primes $p$, since $\text{Sym}^r f$ is cuspidal automorphic for $r = 2,4$ and thus $\sum_{p \leq X} \lambda_{\text{Sym}^r f}(p)/p = O(1)$. When $m \geq 3$ such data for $|\lambda_{\pi}(p)|^4$ is not available unconditionally in general, to the best of the author's knowledge. Assuming the validity of Langlands' functoriality conjecture, a (likely more complicated) expression would follow from the factorization of the standard $L$-function $L(s,f)$ of the representation $f = \pi \otimes \tilde{\pi} \otimes \pi \otimes \tilde{\pi}$, where $\tilde{\pi}$ is the contragredient representation of $\pi$. Using GRC alone, we cheaply obtain the simple upper bound 
\[
\sum_{p \leq X} \frac{|\lambda_{\pi}(p)|^4}{p} \leq m^2 \sum_{p \leq X} \frac{|\lambda_{\pi}(p)|^2}{p} = m^2 \log\log X + O(1),
\]
from Rankin-Selberg theory, and this is the source of the exponent $m^2$ in the range of $h$. \\
We will instead deduce Corollary \ref{cor:RankinSelberg} from Theorem \ref{thm:momCusp} below, which is tailored to $GL_2$ cusp forms.
\end{rem}

\subsubsection{Moments of coefficients of $GL_2$ cusp forms in typical short intervals} \label{sec:RS}
Our next application concerns short interval averages of the moments $n \mapsto |\lambda_f(n)|^{\alpha}$, for any $\alpha > 0$, with the notation of the previous subsection. This generalizes Corollary \ref{cor:RankinSelberg}.
\begin{thm}\label{thm:momCusp}
Let $\alpha > 0$ and define
\[
c_{\alpha} := \frac{2^{\alpha}}{\sqrt{\pi}} \frac{\Gamma\left(\frac{\alpha+1}{2}\right)}{\Gamma(\alpha/2 + 2)}, \quad d_{\alpha} := c_{2\alpha}-2c_{\alpha} + 1.
\]
Let $10 \leq h_0 \leq X/(10 (\log X)^{d_{\alpha}})$ and put $h := h_0 (\log X)^{d_{\alpha}}$. There is a constant $\theta = \theta(\alpha) > 0$ such that 
\[
\frac{1}{X}\int_X^{2X} \left(\frac{1}{h} \sum_{x < n \leq x+h} |\lambda_f(n)|^{\alpha} - \frac{1}{X}\sum_{X < n \leq 2X} |\lambda_f(n)|^{\alpha}\right)^2 dx \ll_{\alpha} \left(\left(\frac{\log\log h_0}{\log h_0}\right)^{c_{\alpha}} + \frac{\log\log X}{(\log X)^{\theta}}\right) (\log X)^{2(c_{\alpha}-1)}.
\]
\end{thm}
When $\alpha \neq 2$, the Rankin-Selberg theory is no longer available. In its place, a crucial role in the proof of this result is played by a quantitative version of the Sato-Tate theorem for non-CM cusp forms, due to Thorner \cite{Tho}, which uses the deep results of Newton and Thorne \cite{NeTh}; see Section \ref{sec:RSProof} for the details.
\begin{rem}
Using the Sato-Tate theorem and \cite[Thm. 1.2.4]{ManThe} it can be shown that $\frac{1}{X}\sum_{X < n \leq 2X} |\lambda_f(n)|^{\alpha} \gg_{\alpha} (\log X)^{c_{\alpha}-1}$, so the estimate in Theorem \ref{thm:momCusp} is indeed non-trivial.
\end{rem}

\subsubsection{Hooley's $\Delta$-function in short intervals}
The distribution of divisors of a typical positive integer is a topic of classical interest, and a source of many difficult problems.
Given an integer $n \in \mb{N}$, let
$$
\mc{D}_n(v) := \frac{1}{d(n)}\sum_{\ss{d|n \\ d \leq e^v}} 1, \quad\quad \text{ for } v \in \mb{R}.
$$
This is a distribution function on the divisors of $n$. 
A concentration function for $\mc{D}_n(v)$, in the sense of probability theory, can be given by
$$
Q(n) := \max_{u \in \mb{R}} |\mc{D}_n(u+1)-\mc{D}_n(u)| = \max_{u \in \mb{R}}\frac{1}{d(n)} \sum_{\ss{d|n \\ e^u < d \leq e^{u+1}}} 1.
$$
Hooley \cite{Hoo} considered the unnormalized variant 
$$
\Delta(n) := d(n) Q(n) = \max_{u \in \mb{R}} \sum_{\ss{d|n \\ e^u < d \leq e^{u+1}}} 1,
$$
now known as Hooley's $\Delta$-function, and used it to attack various problems related, among other things, to inhomogeneous Diophantine approximation by squares, as well as Waring's problem for cubes. Clearly, $0 \leq \Delta(n) \leq d(n)$, but one seeks more refined data about this function.  For example, Erd\H{o}s \cite{ErdHoo} conjectured in 1948 that, except on a set of natural density 0, $\Delta(n) > 1$.\\
Many authors have investigated the average and almost sure behaviour of $\Delta$. Maier and Tenenbaum \cite{MaiTen} proved Erd\H{o}s' conjecture in a quantitative form. A significant portion of Hall and Tenenbaum's book \cite{HallTenBook} is devoted to the $\Delta$ function, including the currently best known upper bound for its mean value (see also \cite{HallTen}). For a partial survey of these results, see \cite{TenHooSur} . \\
Much less has been done concerning the local behaviour of the $\Delta$-function. To the author's knowledge the only result about its short interval behaviour was worked out in the setting of polynomials over a finite field by Gorodetsky \cite[Cor. 1.5]{Gor}.  \\
By relating $\Delta(n)$ to integral averages of the characteristic function of $\mc{D}_n$ (which is multiplicative), we can deduce the following lower bound for $\Delta$ on average over typical short intervals of length $(\log X)^{1/2+\eta}$, for $\eta \in (0,1/2]$.
\begin{cor} \label{cor:Hooley}
Fix $\delta \in (0,1]$, and let $10 \leq h_0 \leq \frac{X}{10(\log X)^{(1+\delta)/2}}$ and set $h = h_0 (\log X)^{(1+\delta)/2}$. Then for all but $o_{h_0 \ra \infty}(X)$ integers $x \in [X/2,X]$ we have
$$
\frac{1}{h}\sum_{x-h < n \leq x} \Delta(n) \gg \delta \log\log X.
$$
\end{cor}

\subsection{Statement of main results}
We fix $B,C \geq 1$, $0 < A \leq B$, and for $X$ large we define $\mc{M}(X;A,B,C)$ to denote the set of multiplicative functions $f : \mb{N} \ra \mb{C}$ such that: 
\begin{enumerate}[(i)]
\item $|f(p)| \leq B$ for all primes $p \leq X$,
\item $|f(n)| \leq d_B(n)^C \text{ for all } n\leq X$,
\item for all $z_0 \leq z\leq w \leq X$, we have
\begin{equation}\label{eq:hyp3}
\sum_{z < p \leq w} \frac{|f(p)|}{p}  \geq A \sum_{z < p \leq w} \frac{1}{p} - O\left(\frac{1}{\log z}\right).
\end{equation}
\end{enumerate}
As described above, the work \cite{MRII} treats 1-bounded multiplicative functions $f \in \mc{M}(X;A,1,1)$. 
We are interested in generalizing the results from \cite{MRII} to be applicable to the collection $\mc{M}(X;A,B,C)$, with $B \geq 1$. For the purpose of applications, we further extend $\mc{M}(X;A,B,C)$ as follows. 

Fixing $\gamma > 0$ and $0 < \sg \leq A$, we define $\mc{M}(X;A,B,C;\gamma,\sg)$ to be the collection of multiplicative functions $f: \mb{N} \ra \mb{C}$ satisfying (i) and (ii), as well as the additional hypotheses
\begin{itemize}
\item[(iii')] for all $z_0 \leq z \leq w \leq X$, we have
\begin{equation}\label{eq:hyp3'}
\sum_{z < p \leq w} \frac{|f(p)|}{p} \geq A \sum_{z < p \leq w} \frac{1}{p} - O\left(\frac{1}{(\log z)^{\gamma}}\right),
\end{equation}
\item[(iv)] letting $t_0 \in [-X,X]$ be a minimizer on $[-X,X]$ of the map
$$
t \mapsto \rho(f,n^{it};X)^2 := \sum_{p \leq X} \frac{|f(p)|-\text{Re}(f(p)p^{-it})}{p},
$$
we have for all $t \in [-2X,2X]$ that
\begin{equation}\label{eq:hyp4}
\rho(f,n^{it};X)^2 \geq \sg \min\{ \log\log X, \log(1+|t-t_0|\log X)\} - O_{A,B}(1).
\end{equation}
\end{itemize}
Condition (iii') is a weaker form of (iii). The full strength of (iii) is needed in \cite[Lem. A.1]{MRII} to obtain (iv) for all $0 < \sg < \sg_A$ with a particular constant $\sg_A > 0$, which is crucial to the proof of \cite[Theorem 1.9]{MRII}. We will show below, as a consequence of \cite[Lem. 5.1(i)]{MRII}, that if  $f \in \mc{M}(X;A,B,C)$ then condition (iv) here holds for any $0 < \sg < \sg_{A,B}$, where 
\begin{equation} \label{eq:sgAB}
\sg_{A,B} := \frac{A}{3} \left(1- \text{sinc}\left(\frac{\pi A}{2B}\right)\right), \quad\quad \text{sinc}(t) := \frac{\sin t}{t} \text{ for } t \neq 0.
\end{equation}
In particular, for any $0 < \sg < \sg_{A,B}$, $\mc{M}(X;A,B,C) \subseteq \mc{M}(X;A,B,C;1,\sg)$. In proving Corollary \ref{cor:RankinSelberg}, for instance, it is profitable to assume (iii') rather than (iii), given currently available quantitative versions of the Sato-Tate theorem (see \eqref{eq:quantST} below). 

In the sequel, fix $B,C \geq 1$, $0 < A \leq B$, $\gamma > 0$ and $0 < \sg \leq A$. We define
\begin{equation}\label{eq:params}
\hat{\sg} := \min\{1,\sg\}, \quad \quad \kappa := \frac{\hat{\sg}}{8B+21}.
\end{equation}
Given $T \geq 1$ we set 
$$
M_f(X;T) := \min_{|t| \leq T} \rho(f,n^{it};X)^2 = \min_{|t| \leq T} \sum_{p \leq X} \frac{|f(p)|-\text{Re}(f(p)p^{-it})}{p}.
$$
We select $t_0(f,T)$ to be a real number $t \in [-T,T]$ that gives the minimum in the definition of $M_f(X;T)$. \\
Finally, for a multiplicative function $f: \mb{N} \ra \mb{C}$ we recall that
$$
H(f;X) := \prod_{p \leq X} \left(1+\frac{(|f(p)|-1)^2}{p}\right),
$$
observing for future reference that whenever $|f(p)| \leq B$ for all $p \leq X$,
\begin{equation} \label{eq:HfXBd}
H(f;X) \asymp_B \prod_{p \leq X} \left(1+\frac{|f(p)|^2-1}{p}\right) \left(1+\frac{|f(p)|-1}{p}\right)^{-2}.
\end{equation}
The main result of this paper is the following.
\begin{thm}\label{thm:MRFull}
Let $X \geq 100$. Let $f \in \mc{M}(X;A,B,C;\gamma,\sg)$, and put $t_0 = t_0(f,X)$. Let $10 \leq h_0 \leq X/10H(f;X)$, and set $h := h_0H(f;X)$. Then
\begin{align*}
&\frac{2}{X}\int_{X/2}^X \left|\frac{1}{h}\sum_{x-h < n \leq x} f(n) - \frac{1}{h} \int_{x-h}^x u^{it_0} du \cdot \frac{2}{X} \sum_{X/2 < n \leq X} f(n)n^{-it_0}\right|^2 dx \\
&\ll_{A,B,C} \left(\left(\frac{\log\log h_0}{\log h_0}\right)^A + \left(\frac{\log\log X}{(\log X)^{\kappa}}\right)^{\min\{1,A\}}\right)\prod_{p \leq X} \left(1+ \frac{|f(p)|-1}{p}\right)^2.
\end{align*}
\end{thm}

\begin{cor}\label{cor:MRVers}
Let $X \geq 100$. Let $f \in \mc{M}(X;A,B,C)$ and put $t_0 = t_0(f,X)$. Let $10 \leq h_0 \leq X/10H(f;X)$, and set $h := h_0H(f;X)$. Then
\begin{align*}
&\frac{2}{X}\int_{X/2}^X \left|\frac{1}{h}\sum_{x-h < n \leq x} f(n) - \frac{1}{h} \int_{x-h}^x u^{it_0} du \cdot \frac{2}{X} \sum_{X/2 < n \leq X} f(n)n^{-it_0}\right|^2 dx \\
&\ll_{A,B,C} \left(\left(\frac{\log\log h_0}{\log h_0}\right)^A + \left(\frac{\log\log X}{(\log X)^{\kappa}}\right)^{\min\{1,A\}}\right)\prod_{p \leq X} \left(1+ \frac{|f(p)|-1}{p}\right)^2,
\end{align*}
for any $0 < \kappa < \kappa_{A,B} := \frac{\min\{1,\sg_{A,B}\}}{16B+21}$.
\end{cor}

\begin{rem} \label{rem:trivBd}
By Shiu's theorem (Lemma \ref{lem:Shiu} below), it is easy to show that the long sum term in the LHS is
$$
\ll \frac{1}{X} \left(\sum_{X/3 < n \leq X} |f(n)|\right)^2 \ll_B \prod_{p \leq X}\left(1+\frac{|f(p)|-1}{p}\right) ^2.
$$
Thus, this theorem shows that the variance is smaller than the square of the ``trivial'' bound for the long sum by a factor tending to 0 provided $h_0(X) \ra \infty$, as $X \ra \infty$.
\end{rem}


\section{Outline of the proof of Theorem \ref{thm:MRFull}}
To prove Theorem \ref{thm:MRFull} we will establish two estimates.
The first compares typical short averages of $f \in \mc{M}(X;A,B,C;\gamma,\sg)$ to typical medium-length ones (i.e., of length $X/(\log X)^c$, for $c = c(\sg) > 0$ small). The techniques involved were developed in \cite{MR}, using certain corresponding refinements that arose in \cite{MRII}.
\begin{thm}\label{thm:MRDB}
Let $B,C \geq 1$, $0 < A \leq B$, $\gamma > 0$ and $0 < \sg \leq A$. Assume that $f \in \mc{M}(X;A,B,C;\gamma,\sg)$. Let $10 \leq h_0 \leq X/10H(f;X)$, set $h_1 := h_0 H(f;X)$ and $h_2 = X/(\log X)^{\hat{\sg}/2}$ and assume that $h_1 \leq h_2$. Finally, put $t_0 = t_0(f,X)$. Then
\begin{align*}
&\frac{2}{X}\int_{X/2}^{X} \left|\frac{1}{h_1} \sum_{x-h_1 < m \leq x} f(m) - \frac{1}{h_1} \int_{x-h_1}^x u^{it_0} du \cdot \frac{1}{h_2} \sum_{x-h_2 < m \leq x} f(m)m^{-it_0} \right|^2 dx \\
&\ll_{A,B,C} \left(\left(\frac{\log\log h_0}{\log h_0}\right)^A + \left(\frac{\log\log X}{(\log X)^{\kappa}}\right)^{\min\{1,A\}}\right)\prod_{p \leq X} \left(1+ \frac{|f(p)|-1}{p}\right)^2.
\end{align*}
\end{thm}
Essential to the treatment of Theorem \ref{thm:MRDB} are strong pointwise upper bounds for Dirichlet polynomials
\[
\sum_{X/3 < n \leq X} \frac{a(n)f(n)}{n^{1+it}},
\]
where $\{a(n)\}_n \subset [0,1]$ is a particular sequence of weights, $f \in \mc{M}(X;A,B,C;\gamma,\sg)$ and $t \in [-X,X]$. To obtain these estimates we will apply some of the recent results about unbounded multiplicative functions described in Section \ref{sec:DBHist}. This is carried out at the beginning of Section \ref{sec:multPart}.

The second estimate we require towards Theorem \ref{thm:MRFull} is a ``Lipschitz'' bound,  approximating the averages of a multiplicative function $f$ on any sufficiently long medium-length interval by a long interval average of $f$. The techniques involved are different from those used in the proof of Theorem \ref{thm:MRDB}, and largely follow the work of Granville, Harper and Soundararajan \cite{GHS}; see Section \ref{sec:Lip} for the details. 
\begin{thm}\label{thm:compLongSums}
Let $B,C \geq 1$, $0 < A \leq B$, $\gamma > 0$ and $0 < \sg \leq A$. Let $X/(\log X)^{\hat{\sg}/2} \leq h \leq X/10$, and let $x \in [X/2,X]$. Assume that $f \in \mc{M}(X;A,B,C;\gamma,\sg)$. Then the following bounds hold:
\begin{align*}
\frac{1}{h}\sum_{x-h < n \leq x}  f(n)n^{-it_0} &= \frac{2}{X}\sum_{X/2 < n \leq X} f(n)n^{-it_0}  + O_{A,B,C}\left( \frac{(\log\log X)^{\hat{\sg}+1}}{(\log X)^{\hat{\sg}/2}} \prod_{p \leq X} \left(1+\frac{|f(p)|-1}{p}\right) \right), \\
\frac{1}{h} \sum_{x-h < n \leq x} f(n) &= \frac{1}{h}\int_{x-h}^x u^{it_0} du \cdot \frac{2}{X}\sum_{X/2 < n \leq X} f(n)n^{-it_0} \\
&+ O_{A,B,C}\left( \frac{(\log\log X)^{\hat{\sg}+1}}{(\log X)^{\hat{\sg}/2}} \prod_{p \leq X} \left(1+\frac{|f(p)|-1}{p}\right) \right).
\end{align*}
\end{thm}

\begin{proof}[Proof of Theorem \ref{thm:MRFull} assuming Theorem \ref{thm:MRDB} and Theorem \ref{thm:compLongSums}]
Assume the hypotheses of Theorem \ref{thm:MRFull}, and set $h' := X/(\log X)^{\hat{\sg}/2}$. If $h > h'$ then Theorem \ref{thm:MRFull} follows immediately from the second estimate in Theorem \ref{thm:compLongSums}. Thus, we may assume that $h \leq h'$.
Applying Cauchy-Schwarz, we get
\begin{align*}
&\frac{2}{X}\int_{X/2}^X\left|\frac{1}{h}\sum_{x-h < n \leq x} f(n) - \frac{1}{h}\int_{x-h}^x u^{it_0} du \cdot \frac{2}{X}\sum_{X/2 < n \leq X} f(n)n^{-it_0}\right|^2 dx \\
&\ll \frac{2}{X}\int_{X/2}^X\left|\frac{1}{h}\sum_{x-h < n \leq x} f(n) - \frac{1}{h}\int_{x-h}^x u^{it_0} du \cdot \frac{1}{h'}\sum_{x-h' < n \leq x} f(n)n^{-it_0}\right|^2 dx \\
&+ \sup_{X/2 < x \leq X} \left|\frac{1}{h'} \sum_{x-h' < n \leq x} f(n)n^{-it_0} - \frac{2}{X} \sum_{X/2 < n \leq X} f(n)n^{-it_0} \right|^2 \\
&=: T_1 + T_2,
\end{align*}
upon trivially bounding $h^{-1} |\int_{x-h}^x u^{it_0} du| \leq 1$. By Theorem \ref{thm:MRDB} and the first estimate of Theorem \ref{thm:compLongSums},
\begin{align*}
T_1 &\ll_{A,B,C} \left(\left(\frac{\log\log h}{\log h}\right)^A + \left(\frac{\log\log X}{(\log X)^{\kappa}}\right)^{\min\{1,A\}}\right) \prod_{p \leq X} \left(1+\frac{|f(p)|-1}{p}\right)^2 \\
T_2 &\ll_{A,B,C} \frac{(\log\log X)^{\hat{\sg}+1}}{(\log X)^{\hat{\sg}/2}} \prod_{p \leq X} \left(1+\frac{|f(p)|-1}{p}\right)^2.
\end{align*}
Combining these bounds proves the claim.
\end{proof}

\section{Averages of Divisor-Bounded Multiplicative Functions and the proof of Theorem \ref{thm:compLongSums}} \label{sec:multPart}
In the sequel we will require control over various averages of multiplicative functions $f \in \mc{M}(X;A,B,C;\gamma,\sg)$. In this and the next subsection, such bounds are derived/recorded. \\
First, we will require some general pointwise estimates for prime power values of $f \in \mc{M}(X;A,B,C;\gamma,\sg)$.
In preparation, define
$$
P(s) := \sum_{\ss{n \geq 1 \\ p^k||n \Rightarrow p^k \leq X}} \frac{f(n)}{n^s} = \prod_{p \leq X} \left(1+\sum_{\ss{k \geq 1 \\ p^k \leq X}} \frac{f(p^k)}{p^{ks}}\right), \quad \text{Re}(s) > 1. 
$$
Wherever $P$ is non-zero we may also write the logarithmic derivative Dirichlet series
$$
-\frac{P'}{P}(s) = \sum_{n \geq 1} \frac{\Lambda_f(n)}{n^s}.
$$
\begin{lem} \label{lem:ppBdswithLambda}
Suppose $f: \mb{N} \ra \mb{C}$ is multiplicative and satisfies $|f(n)| \leq d_B(n)^C$ for all $n \leq X$. \\
a) For any prime power $p^{\nu} \leq X$ we have
$$
|f(p^{\nu})| \ll_{B,C} \begin{cases} \left(\frac{5}{4}\right)^{\nu}\left(1+\frac{\nu}{B-1}\right)^{(B-1)C}, &\text{ if $B > 1$,} \\ 1 &\text{ if $B = 1$.}\end{cases}
$$
b) For any $\eta \in [0,1/2)$ we have
$$
\sum_{\ss{p^{\nu} \leq X \\ \nu \geq 2}} \frac{|f(p^{\nu})|}{p^{(1-\eta)\nu}} \ll_{\eta,B,C} 1.
$$
c) $\Lambda_f(n) = 0$ unless $n = p^{\nu}$ for some prime power $p^{\nu}$. In particular, if $p^{\nu} \leq X$ we have $|\Lambda_f(p)| \leq B \log p$ when $\nu = 1$ and otherwise $|\Lambda_f(p^{\nu})| \ll_{\e,B,C} p^{\e\nu}$.
\end{lem}
\begin{proof}
a) If $B = 1$ then the claim is obvious since $d_B \equiv 1$. Thus, we may assume that $B > 1$. We may also assume that $\nu$ is large relative to $B,C$, for otherwise the estimate is trivial for a suitably large implicit constant.\\  
Given these assumptions, observe that by Stirling's approximation,
\begin{align*}
|f(p^{\nu})| &\leq d_B(p^{\nu})^C = \binom{\nu + B-1}{\nu}^C \ll_{B,C} \left(\frac{\sqrt{2\pi (\nu+B-1)}}{2\pi \sqrt{\nu(B-1)}} \left(1+\frac{B-1}{\nu}\right)^{\nu} \cdot \left(1+\frac{\nu}{B-1}\right)^{B-1}\right)^C \\
&\ll_{B,C} \left(\frac{5}{4}\right)^{\nu} \left(1+\frac{\nu}{B-1}\right)^{C(B-1)},
\end{align*}
provided that $\nu$ is large enough that $\left(1+\frac{B-1}{\nu}\right)^C \leq \frac{5}{4}$. This proves a).\\
b) Let $\delta := \frac{1}{2}-\eta$. For each $2 \leq p \leq X$ we have $p^{1/2} > 5/4$, and thus by a),
$$
\sum_{\ss{\nu \geq 2:  \\ p^{\nu} \leq X}} \frac{|f(p^{\nu})|}{p^{(1-\eta)\nu}} \ll_{B,C} \sum_{\nu \geq 2} \left(1+\frac{\nu}{B-1}\right)^{BC} \left(\frac{5}{4p^{1/2+\delta}}\right)^{\nu} \ll_{B,C,\delta} \sum_{\nu \geq 2} \left(\frac{5}{4p^{(1+\delta)/2}}\right)^{\nu} \ll p^{-1-\delta}.
$$
We deduce b) upon summing over $p \leq X$.\\
c) We begin by giving an expression for $\Lambda_f(p^{\nu})$. \\
In light of a), we may deduce that there is $\sg = \sg(B,C) > 1$ such that when $\text{Re}(s) \geq \sg$,
\begin{equation}\label{eq:smallatPrimes}
\left|\sum_{p^\nu \leq X} \frac{f(p^\nu)}{p^{\nu s}}\right| \leq \frac{1}{2} \text{ for all $2 \leq p \leq X$.}
\end{equation}
It follows from the Euler product representation of $P(s)$ that $P(s) \neq 0$ in the half-plane $\text{Re}(s) \geq \sg$. Thus, $-P'(s)/P(s)$ is also analytic in this half-plane.  \\
Integrating $-P'/P$ term-by-term from $s$ to $\infty$ along a line contained in the half-plane $\text{Re}(s) \geq \sg$, we deduce that
$$
\sum_{n \geq 1} \frac{\Lambda_f(n)}{n^s \log n} = \log P(s) = \sum_{p \leq X} \log\left(1+\sum_{\ss{\nu \geq 1 \\ p^{\nu} \leq X}} \frac{f(p^{\nu})}{p^{\nu s}}\right).
$$
Given \eqref{eq:smallatPrimes}, we obtain the Taylor expansion
\begin{align*}
\sum_{n \geq 1} \frac{\Lambda_f(n)}{n^s\log n} &= \sum_{p \leq X} \sum_{k \geq 1} \frac{(-1)^{k-1}}{k} \sum_{ \ss{\nu_1,\ldots,\nu_k \geq 1 \\ p^{\nu_i} \leq X \forall i}} \frac{f(p^{\nu_1}) \cdots f(p^{\nu_k})}{p^{(\nu_1+\cdots + \nu_k)s}} \\
&= \sum_{\ss{p^{\nu} \\ p \leq X}} \frac{1}{p^{\nu s} \log p^{\nu}}\left(\log p^{\nu} \cdot \sum_{1 \leq k \leq \nu} \frac{(-1)^{k-1}}{k} \sum_{\ss{\nu_1 + \cdots + \nu_k = \nu \\ \nu_1,\ldots,\nu_k \geq 1 \\ p^{\nu_i} \leq X \forall i}} \prod_{1 \leq i \leq k} f(p^{\nu_i})\right).
\end{align*}
By the identity theorem for Dirichlet series, we thus find that $\Lambda_f(n) = 0$ unless $n = p^{\nu}$ for some prime power $p^{\nu}$ with $p \leq X$, in which case
\begin{equation}\label{eq:exact}
\Lambda_f(p^{\nu}) = \log p^{\nu} \cdot \sum_{1 \leq k \leq \nu} \frac{(-1)^{k-1}}{k} \sum_{\ss{\nu_1 + \cdot + \nu_k = \nu \\ \nu_1,\ldots,\nu_k \geq 1 \\ p^{\nu_i} \leq X \forall i}} \prod_{1 \leq i \leq k} f(p^{\nu_i}).
\end{equation}
When $\nu = 1$ we get the expression $\Lambda_f(p) = f(p)\log p$, so that $|\Lambda_f(p)| \leq B \log p$. \\
For $\nu \geq 2$ we simply note using the uniform bound $d_B(n)^C \ll_{B,C,\e} n^{\e}$ and the triangle inequality in \eqref{eq:exact} that
$$
|\Lambda_f(p^{\nu})| \ll_{B,C,\e} \sum_{1 \leq k \leq \nu} \frac{1}{k} \sum_{\ss{\nu_1 + \cdots + \nu_k = \nu \\ \nu_1,\ldots,\nu_k \geq 1}} \prod_{1 \leq i \leq k} p^{\nu_i \e} \leq p^{\nu\e} \mf{p}(\nu),
$$
where $\mf{p}(\nu)$ denotes the number of partitions of the positive integer $\nu$. By a classical bound of Hardy-Ramanujan \cite[Sec. 2.3]{HarRam}, there is an absolute constant $c > 0$ such that 
$$
\mf{p}(\nu) \ll e^{c\sqrt{\nu}} \ll_{\e} p^{\nu \e},
$$
which implies the claim.
\end{proof}

We will use the following upper bound for non-negative functions repeatedly in the sequel.
\begin{lem}[P. Shiu; \cite{Shiu}, Thm. 1] \label{lem:Shiu}
Let $f: \mb{N} \ra \mb{C}$ be a multiplicative function satisfying $|f(n)| \leq d_B(n)^C$ for all $n \leq X$. Let $\sqrt{X} < Y \leq X$, $\delta \in (0,1)$ and let $Y^{\delta} \leq y \leq Y$. Then
$$
\sum_{Y-y < n \leq Y} |f(n)| \ll_{B,C,\delta} y \mc{P}_f(X).
$$
\end{lem}
\begin{proof}
The hypotheses required to apply \cite[Thm. 1]{Shiu} are more precisely that $|f(n)| \ll_{\e} n^{\e}$ for all $n \leq X$, and that there is a constant $A \geq 1$ such that $|f(p^{\nu})| \leq A^{\nu}$ for all $p^{\nu} \leq X$. 
The first hypothesis is obvious from $d_B(n)^C \leq d(n)^{\lceil B\rceil C} \ll_{B,C,\e} n^{\e}$, while the second is implied by Lemma \ref{lem:ppBdswithLambda} a).
\end{proof}

\begin{lem} \label{lem:HalType}
Let $f \in \mc{M}(X;A,B,C;\gamma,\sg)$ and let $t_0 = t_0(f,X)$. Let $X^{1/5} \leq Y \leq X$, and let $2 \leq P \leq Q \leq \exp\left(\frac{\log X}{\log\log X}\right)$. Then for any $1 \leq Z \leq \log X$,
$$
\sup_{Z < |u| \leq X/2} \left|\sum_{\ss{ n \leq Y \\ p|n \Rightarrow p \notin [P,Q]}} f(n)n^{-i(t_0+u)}\right| \ll_{A,B,C} Y \mc{P}_f(X) \left( \left(\frac{\log Q}{\log P}\right)^{2B}\frac{\log\log X}{\log X^{\sg}} + \frac{1}{\sqrt{Z}} \right).
$$
\end{lem}
\begin{proof}
Define $\beta(n) := f(n)1_{p|n \Rightarrow p \notin [P,Q]}$, and for $t \in \mb{R}$ set $\beta_{t}(n) := \beta(n)n^{-it}$. Note that $|\beta_t(n)| \leq |f(n)|$ for all $n$ and $t \in \mb{R}$. As $f \in \mc{M}(X;A,B,C;\gamma,\sg)$, we have that
\begin{enumerate}
\item $\max_{p \leq X} |f(p)|\leq B$ 
\item $\sum_{\ss{p^k\leq X \\ k \geq 2}} \frac{|f(p^k)|\log p^k}{p^k} \ll_{B,C} 1$ by Lemma \ref{lem:ppBdswithLambda}b), and
\item $\sum_{y < p \leq X} \frac{|f(p)|}{p} \geq A \log\left(\frac{\log X}{\log y}\right) - O_{A,B}(1)$ uniformly in $2 \leq y \leq X$,
\end{enumerate}
for all $t \in \mb{R}$. Thus, the hypotheses of \cite[Cor. 2.1]{TenVM} are fulfilled with $r = |f|$. Applying that result gives, for every $Z < |u| \leq X/2$,
$$
\left|\sum_{n \leq Y} \beta_{t_0+u}(n) \right| \ll_{A,B,C} \left(\sum_{n \leq Y} |f(n)|\right) \left((1+M_{\beta_{t_0+u}(Y;Z/2)}e^{-M_{\beta_{t_0+u}}(Y;Z/2)} + \frac{1}{\sqrt{Z}}\right).
$$
Let $t = t(u) \in [-Z/2,Z/2]$ be the minimizer implicit in $M_{f_{t_0+u}}(Y;Z/2)$, so that 
$$
M_{\beta_{t_0+u}}(Y;Z/2) = \rho(\beta,n^{i(t_0+u+t)};Y)^2 \leq 2B \log\log X.
$$ 
As $X^{1/5} \leq Y \leq X$,
$$
\rho(\beta,n^{i(t_0+u+t)}; Y)^2 = \rho(\beta,n^{i(t_0+u+t)}; X)^2 - O_B(1) \geq \rho(f,n^{i(t_0+u+t)}; X)^2 - 2B \log\left(\frac{\log Q}{\log P}\right) - O_B(1).
$$
Since $f \in \mc{M}(X;A,B,C;\gamma,\sg)$ and $|t_0 + u+t| \leq 2X$ and $|u+t| > Z/2$, we have by \eqref{eq:hyp4} that
$$
\rho(f,n^{i(t_0+u+t)};X)^2 \geq \sg \min\{\log\log X, \log(1+|u+t|\log X)\} - O_{A,B}(1) \geq \sg \log\log X - O_{A,B}(1).
$$
It thus follows that
$$
\max_{Z < |u| \leq X/2} \left|\sum_{n \leq Y} \beta_{t_0+u}(n) \right| \ll_{A,B,C} \left(\sum_{n \leq Y} |f(n)|\right) \left( \left(\frac{\log Q}{\log P} \right)^{2B} \frac{\log\log X}{(\log X)^{\sg}} + \frac{1}{\sqrt{Z}}\right).
$$
Finally, applying Lemma \ref{lem:Shiu} together with Mertens' theorem, we obtain
$$
\sum_{n \leq Y} |f(n)| \ll_{B,C} Y \prod_{p \leq Y} \left(1+ \frac{|f(p)|-1}{p}\right) \ll_B Y \mc{P}_f(X),
$$
and the claim follows.
\end{proof}

We need the following estimate for certain divisor-bounded functions on $y$-smooth\footnote{By a \emph{$y$-smooth} or \emph{$y$-friable} integer we mean a positive integer $n$ such that $p|n \Rightarrow p \leq y$.} integers, which is essentially due to Tenenbaum and Wu \cite[Cor. 2.2]{TeWu}, improving on work of Song \cite{Song}. An important r\^{o}le is played by  the function $\rho_k(u)$ for $k \in \mb{N}$ and $u \geq 0$, which is a generalization of the classical Dickman-de Bruijn function, defined by the differential delay equation
$$
u\rho_k'(u) = (k-1)\rho_k(u)-k\rho_k(u-1) \text{ if } u \geq 1,
$$
and $\rho_k(u) := u^{k-1}/\Gamma(k)$ for $0\leq u < 1$.
\begin{lem} \label{lem:Song}
Let $g: \mb{N} \ra \mb{R}$ be a non-negative multiplicative function for which there are real constants $\delta \in (0,1)$, $\eta \in (0,1/2)$ and $D > 0$, and an integer $k \geq 1$ such that
\begin{align}
&\sum_{p \leq z} g(p)\log p = k z + O(z/(\log z)^{\delta}) \text{ for all } z \geq 2, \label{eq:asympKap} \\
&\sum_{p^{\nu}, \nu \geq 2} \frac{g(p^{\nu})}{p^{(1-\eta)\nu}} \leq D. \label{eq:ppBd}
\end{align}
Let $x \geq 3$ and let $\exp\left((\log x \log\log x)^{2/(2+\delta)}\right) \leq y \leq x$. Set $u := \frac{\log x}{\log y}$. Then
$$
\sum_{\ss{n \leq x \\ P^+(n) \leq y}} g(n)  = e^{-\gamma k} x \rho_{k}(u) \frac{G(1,y)}{\log y} \left(1+O\left(\frac{\log(u+1)}{(\log y)^{\delta/2}}\right)\right),
$$
where $G(s,y) := \sum_{P^+(n) \leq y} g(n)n^{-s}$ for $\text{Re}(s) > 0$, and $P^+(n)$ denotes the largest prime factor of $n$. \\
In particular, for $y$ in the given range and such that $u \ra \infty$ we have
$$
\sum_{\ss{n \leq x \\ P^+(n) \leq y}} g(n) \ll_{k,D,\delta} x(\log y)^{k-1} \exp\left(-\frac{1}{3}u\log u \right).
$$
\end{lem}
\begin{proof}
The first claim is a special case of \cite[Cor. 2.2]{TeWu}. \\
For the second claim, we note that
$$
G(1,y) 
\leq \exp\left(\sum_{p \leq y} \frac{g(p)}{p} + \sum_{\ss{p \leq y \\ \nu \geq 2}} \frac{g(p^{\nu})}{p^{\nu}}\right) \ll_D \exp\left(\sum_{p \leq y} \frac{g(p)}{p}\right) \ll_{\delta} (\log y)^{k},
$$
where the penultimate estimate follows from \eqref{eq:ppBd}, and the last estimate is obtained by partial summation from \eqref{eq:asympKap}. Furthermore, by \cite[(3.10)]{Smida} and well-known upper bounds for the Dickman-de Bruijn function (e.g., \cite[(1.6)]{GraSmooth}), we have
$$
\rho_k(u) = k^{u+O(u/\log(1+u))} \rho(u) \leq \exp\left(2u \log k - \frac{1}{2}u\log u \right) \leq \exp\left(-\frac{1}{3}u\log u\right),
$$
whenever $u$ is large enough in terms of $k$, and the claim follows.
\end{proof}

By combining the last two lemmas, we may deduce the following upper bound for Dirichlet polynomials of a special type (cf. \cite[Lem. 3]{MR}).
\begin{cor}\label{cor:Hal}
Let $10 \leq P \leq Q \leq \exp\left(\frac{\log X}{\log\log X}\right)$, and let $1 \leq Z \leq \log X$. Assume the hypotheses of Lemma \ref{lem:HalType}, and assume $X \geq X_0(B,C)$. Then for any $\sqrt{X} \leq Y \leq X$,
\begin{align*}
&\sup_{Z < |u| \leq X/2} \left|\sum_{Y/3 < n \leq Y} \frac{f(n)}{n^{1+i(t_0+u)}(1+\omega_{[P,Q]}(n))}\right| \\
&\ll_{A,B,C} \mc{P}_{f}(X) \left( \left(\frac{\log Q}{\log P}\right)^{3B}\frac{\log\log X}{(\log X)^{\sg}} + \left(\frac{\log Q}{\log P}\right)^{B}\frac{1}{\sqrt{Z}}\right),
\end{align*}
where $\omega_{[P,Q]}(n) := \sum_{\ss{p|n \\ P \leq p \leq Q}} 1$.
\end{cor}
\begin{proof}
Fix $u \in [-Z,Z]$ and set $t := t_0 + u$.  Write $f = \alpha \ast \beta$, where $\alpha$ and $\beta$ are multiplicative functions with $\alpha(p^k) = f(p^k)$ whenever $P \leq p \leq Q$ and $p^k \leq X$, and $\beta(p^k) = f(p^k)$ for all other primes powers $p^k \leq X$. We apply the hyperbola method with $M = \sqrt{Y}$ to get
\begin{align*}
\left|\sum_{Y/3 < n \leq Y} \frac{f(n)}{n^{1+it}(1+\omega_{[P,Q]}(n))}\right| &\ll \sum_{a \leq M} \frac{|\alpha(a)|}{a} \left|\sum_{Y/(3a) < b \leq Y/a} \frac{\beta(b)b^{-it}}{b}\right| + \sum_{b \leq Y/M} \frac{1}{b}\sum_{\max\{M,Y/(3b)\} < a \leq Y/b} \frac{|\alpha(a)|}{a} \\
&=: \mc{R}_1 + \mc{R}_2.
\end{align*}
To bound $\mc{R}_1$ we apply partial summation and Lemma \ref{lem:HalType} to obtain, uniformly in $u$,
\begin{align*}
\sum_{Y/(3a) < b \leq Y/a} \frac{\beta(b)b^{-it}}{b} &\ll_{A,B,C} \sup_{Y/(3a) \leq y \leq Y/a} \frac{1}{y} \left|\sum_{n \leq y} \beta(b)b^{-it}\right| \\
&\ll \mc{P}_f(X)\left(\left(\frac{\log Q}{\log P}\right)^{2B} \frac{\log\log X}{(\log X)^{\sg}} + \frac{1}{\sqrt{Z}}\right).
\end{align*}
Given the prime power support of $\alpha$, we have 
$$
\sum_{a \leq M} \frac{|\alpha(a)|}{a} \ll_{B,C} \prod_{P \leq p \leq Q} \left(1+\frac{|f(p)|}{p}\right) \ll_B \left(\frac{\log Q}{\log P}\right)^B,
$$
so that on combining this with the previous estimate, we obtain
$$
\mc{R}_1 \ll_{A,B,C} \mc{P}_{f}(X) \left(\left(\frac{\log Q}{\log P}\right)^{3B} \frac{\log\log X}{(\log X)^{\sg}} + \left(\frac{\log Q}{\log P}\right)^{B}\frac{1}{\sqrt{Z}}\right).
$$
Next, consider $\mc{R}_2$. Note that $\alpha(n) = f(n)1_{p|n \Rightarrow p \in [P,Q]}$, and so since $|f(n)| \leq d_{B}(n)^C$ uniformly over $n \leq X$ we have
$$
\sum_{X/(3b) < a \leq X/b} \frac{|\alpha(a)|}{a} \ll \frac{b}{X} \sum_{\ss{n \leq X/b \\ P^+(n) \leq Q}} g(n),
$$
where $g(n) := d_{\lceil B\rceil}(n)^{\lceil C\rceil}$. Note that $g(n)$ takes integer values, and in particular taking $k := \lceil B\rceil^{\lceil C\rceil} \in \mb{Z}$, we have
$$
\sum_{p \leq X} g(p)\log p = k \sum_{p \leq X} \log p = k X + O(X/(\log X)^{1/2}),
$$
say, by the prime number theorem. Furthermore, that $g$ satisfies \eqref{eq:ppBd} with some $\eta \in (0,1/2)$ and $D = O_{B,C}(1)$ is the content of Lemma \ref{lem:ppBdswithLambda}b). Hence, as $u_b := \log (X/b)/\log Q \geq \frac{\log X}{2\log Q} \geq \frac{1}{2}\log\log X$, Lemma \ref{lem:Song} implies that when $X$ is sufficiently large in terms of $B$ and $C$ we obtain
$$
\frac{b}{X} \sum_{\ss{n \leq X/b \\ P^+(n) \leq Q}} g(n) \ll (\log Q)^{k-1} \exp\left(-\frac{1}{6} u_b \log u_b \right) \ll_{B,C} (\log X)^{-200}.
$$
Combining this with the bound
$$
\sum_{b \leq X/M} \frac{|\beta(b)|}{b} \leq \sum_{b \leq X/M} \frac{|f(b)|}{b} \ll (\log X) \mc{P}_{f}(X),
$$
which again follows from partial summation and Lemma \ref{lem:Shiu}, we obtain $\mc{R}_2 \ll (\log X)^{-100}$.
Altogether, we conclude that
$$
\left|\sum_{Y/3 < n \leq Y} \frac{f(n)}{n^{1+i(t_0+u)}(1+\omega_{[P,Q]}(n))}\right| \ll_{B,C} \mc{P}_{f}(X) \left( (\log X)^{-100} + \left(\frac{\log Q}{\log P}\right)^{3B} \frac{\log\log X}{(\log X)^{\sg}} + \left(\frac{\log Q}{\log P}\right)^B \frac{1}{\sqrt{Z}}\right),
$$
uniformly over $Z < |u| \leq X/2$, and the claim follows since $Z^{-1/2} \geq (\log X)^{-1/2}$.
\end{proof}

\subsection{Lipschitz Bounds and Main Terms} \label{sec:Lip}
In this subsection we derive a slight refinement of the Lipschitz bounds found in \cite[Thm. 1.5]{GHS}. Specifically, our estimates are sensitive to the distribution of values $|f(p)|$, which will allow us to obtain Theorem \ref{thm:compLongSums}. See also \cite{Mat} for some related Lipschitz-type bounds for unbounded multiplicative functions that share some overlap with the general result obtained presently. 
\begin{thm}[Relative Lipschitz bounds] \label{thm:Lipschitz}
Let $1 \leq w \leq X^{1/3}$ and let $f \in \mc{M}(X;A,B,C;\gamma,\sg)$. Set $t_0 = t_0(f,X)$. Then
\begin{align*}
&\left|\frac{w}{X} \sum_{n \leq X/w} f(n)n^{-it_0} - \frac{1}{X} \sum_{n \leq X} f(n)n^{-it_0}\right| \ll_{A,B,C} \log\left(\frac{\log X}{\log(ew)}\right) \left(\frac{\log(ew) + \log\log X}{\log X}\right)^{\hat{\sg}}\mc{P}_f(X),
\end{align*}
where $\hat{\sg} := \min\{1,\sg\}$. The same bound holds for the quantity
$$
\left|\left(\frac{w}{X}\right)^{1+it_0} \sum_{n \leq X/w} f(n) - \frac{1}{X^{1+it_0}} \sum_{n \leq X} f(n)\right|.
$$
\end{thm}
To this end, we need to introduce some notation that is consistent with the notation from \cite{GHS}. For $\text{Re}(s) > 1$ we write
$$
\mc{F}(s) = \sum_{\ss{n \geq 1 \\ p^k||n \Rightarrow p^k \leq X}} \frac{f(n)n^{-it_0}}{n^s}.
$$
For each prime power $p^k \leq X$, $k \geq 1$, define
$$
s(p^k) := \begin{cases} f(p^k)p^{-ikt_0} &\text{ if $p \leq y$} \\ 0 &\text{ if $p > y$} \end{cases} \quad\quad
\ell(p^k) := \begin{cases} 0 &\text{ if $p \leq y$} \\ f(p^k)p^{-ikt_0} &\text{ if $p > y$} \end{cases},
$$
where $y\geq 2$ is a large parameter to be chosen later. We extend $s$ and $\ell$ multiplicatively to all $n \in \mb{N}$ with $p^k||n \Rightarrow p^k \leq X$, and set $s(n) = \ell(n) = 0$ otherwise. For $\text{Re}(s) > 1$, also define
$$
\mc{S}(s) = \sum_{n \geq 1} \frac{s(n)}{n^s}, \quad \quad \mc{L}(s) := \sum_{n \geq 1} \frac{\ell(n)}{n^s}.
$$
We recall that $\Lambda_{\ell}(n)$ is $n$th coefficient of the Dirichlet series $-\mc{L}'/\mc{L}(s)$, the logarithmic derivative of $\mc{L}(s)$, which is well-defined for all $\text{Re}(s) > 1$ whenever $y \geq y_0(B,C)$ by Lemma \ref{lem:ppBdswithLambda} c).

\begin{lem} \label{lem:LfnBd}
Let $f \in \mc{M}(X;A,B,C;\gamma,\sg)$. Let $t \in \mb{R}$ and $\xi \geq 1/\log X$. Then
$$
|\mc{F}(1+\xi + it)| \ll_{A,B,C} \xi^{-1} (1+\xi \log X)^{1-A} \mc{P}_f(X) e^{-\rho(fn^{-it_0},n^{it};e^{1/\xi})^2} \ll_B (\log X)^B.
$$
\end{lem}
\begin{proof}
By Lemma \ref{lem:ppBdswithLambda}b) and $\max_{p \leq X} |f(p)|\leq B$, we deduce that
\begin{align*}
|\mc{F}(1+\xi+it)| &\ll_{B} \prod_{p \leq X} \left|1+ \frac{f(p)p^{-i(t_0+t)}}{p^{1+\xi}}\right| \exp\left(\sum_{\ss{p^k \leq X \\ k \geq 2}} \frac{|f(p^k)|}{p^{k(1+\xi)}}\right) \\
&\ll_{B,C} \prod_{p\leq X} \left(1+\frac{|f(p)|}{p^{1+\xi}}\right)\left(1+ \frac{\text{Re}(f(p)p^{-i(t_0+t)})-|f(p)|}{p^{1+\xi}}\right).
\end{align*}
By partial summation and the prime number theorem, the estimates
\begin{align}\label{eq:tails}
\sum_{p > e^{1/\xi}} \frac{\alpha_p}{p^{1+\xi}} &\ll_B \int_{e^{1/\xi}}^{\infty} e^{-\xi v} \frac{dv}{v} \ll_B 1 \nonumber \\
\sum_{p \leq e^{1/\xi}} \left(\frac{\alpha_p}{p}-\frac{\alpha_p}{p^{1\pm \xi}}\right) &\ll B\xi \sum_{p \leq e^{1/\xi}} \frac{\log p}{p} \ll_B 1,
\end{align}
are valid for any sequence $\{\alpha_p\}_p \subset \mb{C}$ with $\max_p |\alpha_p| \ll_B 1$. It follows that
\begin{align*}
|\mc{F}(1+\xi+it)| &\ll_{B,C} \prod_{p \leq e^{1/\xi}} \left(1+\frac{|f(p)|}{p}\right) \exp\left(-\sum_{p \leq e^{1/\xi}} \frac{|f(p)|-\text{Re}(f(p)p^{-i(t_0+t)})}{p}\right) \\
&\ll_B \xi^{-1} \mc{P}_f(X)  \exp\left(\sum_{e^{1/\xi} < p \leq X} \frac{1-|f(p)|}{p}\right) e^{-\rho(fn^{-it_0},n^{it}; e^{1/\xi})^2}.
\end{align*}
Since $f \in \mc{M}(X;A,B;\gamma,\sg)$, we have
$$
\sum_{e^{1/\xi} < p \leq X} \frac{1-|f(p)|}{p} \leq (1-A) \sum_{e^{1/\xi} < p \leq X} \frac{1}{p} +O(\xi^{\gamma}) = (1-A) \log(\xi \log X) + O_A(1).
$$
We thus obtain
$$
|\mc{F}(1+\xi+it)| \ll_{A,B,C} \xi^{-1} (1+\xi \log X)^{1-A} \mc{P}_f(X)e^{-\rho(fn^{-it_0},n^{it}; e^{1/\xi})^2},
$$
which proves the first claimed estimate.\\
To obtain the second, note that $\rho(fn^{-it_0},n^{it};Y)^2 \geq 0$ for all $Y \geq 2$, and so using $|f(p)| \leq B$ and $A \geq 0$ we obtain the further bound
\begin{align*}
\ll \xi^{-1}(\xi\log X)^{1-A} \mc{P}_f(X) \ll (\log X) \mc{P}_f(X) &\ll_B \exp\left(\sum_{p \leq X} \frac{|f(p)|}{p} \right) \ll (\log X)^B,
\end{align*} 
as claimed.
\end{proof}

To bound certain error terms in the proof of Theorem \ref{thm:Lipschitz} we require the following estimate, whose proof largely follows that of \cite[Lem. 2.4]{GHS}
\begin{lem}\label{lem:24ref}
Let $1 \leq w \leq X^{1/3}$, $w \leq y \leq \sqrt{X}$ and $\eta := 1/\log y$. Then for any $X/w \leq Z \leq X$,
$$
\sum_{mn \leq Z} |s(m)|\frac{|\ell(n)|}{n^{\eta}} + \int_0^{\eta}\sum_{mkn \leq Z} |s(m)|\frac{|\Lambda_{\ell}(k)||\ell(n)|}{k^{\alpha}n^{2\eta + \alpha}}d\alpha \ll_{A,B,C} Z\left(\frac{\log y}{\log Z}\right)^A \mc{P}_f(X).
$$
\end{lem}
\begin{proof}
In the first sum the summands arise from a Dirichlet convolution of multiplicative functions $|s| \ast |\ell|n^{-\eta}$, and we clearly have $|s(n)|,|\ell(n)| \leq d_B(n)^C$ for all $n \leq X$. By Lemma \ref{lem:Shiu} and \eqref{eq:tails}, the first sum is therefore
$$
\ll_{B,C} \frac{Z}{\log Z} \exp\left(\sum_{p \leq y} \frac{|f(p)|}{p} + \sum_{y < p \leq Z} \frac{|f(p)|}{p^{1+\eta}}\right) \ll_B Z\frac{\log y}{\log X} \mc{P}_f(X) \exp\left(\sum_{y < p \leq X} \frac{1-|f(p)|}{p}\right).
$$
Arguing as in the previous lemma, since $f \in \mc{M}(X;A,B,C;\gamma,\sg)$ this is bounded by
$$
\ll_A Z\frac{\log y}{\log X} \cdot \left(\frac{\log X}{\log y}\right)^{1-A} \mc{P}_f(X) \ll Z\left(\frac{\log y}{\log X}\right)^A \mc{P}_f(X).
$$
For the second term, we use Lemma \ref{lem:ppBdswithLambda}c), which shows that$|\Lambda_{\ell}(p)| \leq B \log p $ and $|\Lambda_{\ell}(p^{\nu})| \ll_{\e,B,C} p^{\nu\e}$. It follows that for $1 \leq K \leq X$,
$$
\left|\sum_{k \leq K} \Lambda_{\ell}(n)n^{-\alpha}\right| \leq B\sum_{p \leq K} p^{-\alpha} \log p + O_{B,C} \left(\sum_{\ss{p^{\nu} \leq K \\ \nu \geq 2}} p^{\nu/6} \right) \ll_{B,C} K^{1-\alpha} + K^{2/3},
$$
say, which is acceptable.
%
%
Therefore taking $K = Z/mn$, the $\alpha$ integral in the statement is
$$
\ll_{B,C} \int_0^{\eta} Z^{1-\alpha}\sum_{mn \leq Z} \frac{|s(m)|}{m^{1-\alpha}}\frac{|\ell(n)|}{n^{1+2\eta}} d\alpha.
$$
Extending the inner sum by positivity to all products $mn$ with $p^k||mn \Rightarrow p^k \leq Z$ and using the estimates \eqref{eq:tails} once again, we may bound the integral using Euler products as
\begin{align*}
&\ll_{B,C} \int_0^{\eta} Z^{1-\alpha} \cdot \prod_{p \leq y} \left(1+\frac{|f(p)|}{p^{1-\alpha}}\right) \prod_{y < p \leq X} \left(1+\frac{|f(p)|}{p^{1+2\eta}}\right) d\alpha \\
&\ll_B Z\prod_{p \leq y} \left(1+\frac{|f(p)|}{p}\right) \int_0^{\eta} Z^{-\alpha} d\alpha \ll_B Z \frac{\log y}{\log Z} \mc{P}_f(X)\exp\left(\sum_{y < p \leq X} \frac{1-|f(p)|}{p}\right) \\
&\ll_{A,B} Z\left(\frac{\log y}{\log X}\right)^A\mc{P}_f(X).
\end{align*}
This completes the proof.
\end{proof}

\begin{proof}[Proof of Theorem \ref{thm:Lipschitz}]
The proof follows that of \cite[Thm. 1.5]{GHS}, and we principally highlight the differences. \\
We begin with the first claim of the theorem. Let $T := (\log X)^{B+1}$ and $y := \max\{ew,T^2\}$. 
Fix $\eta := 1/\log y$, $c_0 := 1+1/\log X$. Then \cite[Lem. 2.2]{GHS} (replacing $\beta$ by $\beta/2$) and Lemma \ref{lem:24ref} combine to show that
\begin{align*}
&\frac{1}{X}\sum_{n \leq X} f(n)n^{-it_0} - \frac{w}{X}\sum_{n \leq X/w} f(n)n^{-it_0} \\
&= \int_0^{\eta}\int_0^{\eta} \frac{1}{\pi i} \int_{c_0-i\infty}^{c_0+i\infty} \mc{S}(s) \mc{L}(s+\alpha) \frac{\mc{L}'}{\mc{L}}(s+\alpha)\frac{\mc{L}'}{\mc{L}}(s+\alpha+2\beta) \frac{X^{s-1} (1-w^{1-s})}{s} ds d\beta d\alpha \\
&+ O_{A,B,C}\left(\left(\frac{\log y}{\log X}\right)^A\mc{P}_f(X)\right).
\end{align*}
Consider the inner integral over $s$. Shifting $s \mapsto s - \alpha-\beta$ and applying \cite[Lem. 2.5]{GHS}, this is
\begin{align*}
\frac{1}{\pi i} \int_{c_0-iT}^{c_0+iT} \mc{S}(s-\alpha-\beta) \mc{L}(s+\beta) \left(\sum_{y < m < X/y} \frac{\Lambda_{\ell}(m)}{m^{s-\beta}}\right)\left(\sum_{y < n < X/y} \frac{\Lambda_{\ell}(n)}{n^{s+\beta}}\right) &\frac{X^{s-1-\alpha-\beta}(1-w^{1+\alpha + \beta-s})}{s-\alpha-\beta} ds \\
&+ O\left(\frac{1}{\log X}\right).
\end{align*}
Extracting the maximum over $|t| \leq T$, then applying Cauchy-Schwarz and \cite[Lem. 2.6]{GHS}, the main term for the $s$-integral is bounded above by
\begin{align*}
&\ll X^{-\alpha-\beta} \left(\max_{|t| \leq T} \frac{|\mc{S}(c_0-\alpha-\beta+it)\mc{L}(c_0+\beta+it)||1-w^{1+\alpha+\beta-c_0-it}|}{|c_0-\alpha-\beta+it|}\right) \\
&\cdot \left(\int_{-T}^{T} \left|\sum_{y < m < X/y} \frac{\Lambda_{\ell}(m)}{m^{c_0-\beta-it}}\right|^2 dt\right)^{1/2}\left(\int_{-T}^{T} \left|\sum_{y < m < X/y} \frac{\Lambda_{\ell}(m)}{m^{c_0+\beta-it}}\right|^2 dt\right)^{1/2} \\
&\ll_{B,C} X^{-\alpha-\beta} \left(\max_{|t| \leq T} \frac{|\mc{S}(c_0-\alpha-\beta+it)\mc{L}(c_0+\beta+it)||1-w^{1+\alpha+\beta-c_0-it}|}{|c_0-\alpha-\beta+it|}\right) \\
&\cdot \left(\sum_{y < p < X/y} \frac{\log p}{p^{c_0 -2\beta}} + y^{-1/2+2\beta } \right)^{1/2}\left(\sum_{y < p < X/y} \frac{\log p}{p^{c_0 +2\beta}} + y^{-1/2-2\beta}\right)^{1/2} \\
&\ll X^{-\alpha-\beta} \left(\frac{X}{y}\right)^{\beta}\min\{\log X, 1/\beta\} \left(\max_{|t| \leq T} \frac{|\mc{S}(c_0-\alpha-\beta+it)\mc{L}(c_0+\beta+it)||1-w^{1+\alpha+\beta-c_0-it}|}{|c_0-\alpha-\beta+it|}\right).
\end{align*}
Furthermore, by \eqref{eq:tails} and $\alpha,\beta \leq \eta = 1/\log y$, we see that for any $t \in \mb{R}$,
\begin{align*}
X^{-\alpha-\beta}\left(\frac{X}{y}\right)^{\beta}|\mc{S}(c_0-\alpha-\beta+it)\mc{L}(c_0+\beta+it)| &\ll_{B,C} X^{-\alpha}y^{-\beta}|\mc{S}(c_0+\beta+it)\mc{L}(c_0+\beta+it)| \\
&\ll X^{-\alpha} |\mc{F}(c_0+\beta+it)|.
\end{align*}
Thus, we have so far shown that
\begin{align}
&\left|\frac{1}{X}\sum_{n \leq X} f(n)n^{-it_0} - \frac{w}{X}\sum_{n \leq X/w} f(n)n^{-it_0}\right| \nonumber\\
&\ll_{A,B,C} \int_0^{\eta}\int_0^{\eta} X^{-\alpha} \min\{\log X,1/\beta\} \max_{|t| \leq T} \frac{|\mc{F}(c_0+\beta+it)(1-w^{1+\alpha+\beta-c_0-it})|}{|c_0+\beta+it|} d\beta d\alpha \label{eq:intstoBd}\\
&+ \left(\frac{\log\log X + \log(ew)}{\log X}\right)^A\mc{P}_f(X) + \frac{1}{\log X}. \nonumber
\end{align}
Observe next that
$$
|w^{-\beta-it} - w^{1+\alpha+\beta-c_0-it}| \ll \left(\alpha + \beta + \frac{1}{\log X}\right) \log (ew),
$$
so that we may rewrite the integral expression in \eqref{eq:intstoBd} as
\begin{align*}
&\int_0^{\eta} \left(\int_0^{\eta} X^{-\alpha} dx\right) \min\{\log X,1/\beta\} \max_{|t| \leq T} \frac{|\mc{F}(c_0+\beta+it)(1-w^{-\beta-it})|}{|c_0+\beta+it|} d\beta \\
&+ (\log (ew))\int_0^{\eta}\max_{|t| \leq T} |\mc{F}(c_0+\beta+it)| \min\{\log X , \beta^{-1}\}\int_0^{\eta} X^{-\alpha} \left(\alpha+\beta + 1/\log X\right) d\alpha d\beta \\
&=: T_1 + T_2.
\end{align*}
We first estimate $T_2$. The integral over $\alpha$ is
$$
\ll \left(\beta + \frac{1}{\log X}\right) \int_0^{\eta} X^{-\alpha}d\alpha + \int_0^{\eta} \alpha X^{-\alpha} d\alpha \ll \frac{1}{\log X} \left(\beta+ \frac{1}{\log X}\right).
$$
Applying Lemma \ref{lem:LfnBd} (with $\xi = 1/\log X + \beta$) and $\rho(fn^{-it_0},n^{it};Y)^2 \geq 0$ for all $Y \geq 2$, 
\begin{align*}
T_2 &\ll_{A,B,C} (\log (ew))\mc{P}_f(X)\int_0^{\eta} \min\{1,(\beta \log X)^{-1}\}(1+\beta\log X)^{1-A} d\beta.
\end{align*}
Splitting the $\beta$-integral at $1/\log X$ and evaluating, we obtain
\begin{align*}
T_2 
&\ll_{A,B,C} (\log (ew))\mc{P}_f(X) \left(\frac{1}{\log X} + \frac{1_{A = 1} \log(\eta \log X) + 1}{(\log X)^A}\left((\log X)^{A-1} + \eta^{1-A}\right)\right) \\
&\ll_B \mc{P}_f(X)\left(\frac{(\log (ew))(1+1_{A = 1} \log(\log X/\log(ew)))}{\log X} + \left(\frac{\log (ew) + \log\log X}{\log X}\right)^A\right) \\
&\ll  \mc{P}_f(X) \left(\frac{\log (ew) + \log\log X}{\log X}\right)^{\min\{1,A\}} \left(1+1_{A = 1} \log\left(\frac{\log X}{\log(ew)}\right)\right).
\end{align*}
We now turn to $T_1$. By evaluating the $\alpha$ integral, we have
$$
T_1 \ll \frac{1}{\log X} \int_0^{\eta} \min\{\log X,1/\beta\} \max_{|t| \leq T} \frac{|\mc{F}(c_0+\beta+it)||1-w^{-\beta-it}|}{|c_0+\beta+it|} d\beta.
$$
Put $T' := \frac{1}{2}(\log X)^B$. If the maximum occurs at $|t| > T'$ then using the second estimate in Lemma \ref{lem:LfnBd} we obtain
$$
T_1 \ll_{B,C} \frac{1}{\log X} \int_0^{\eta} \min\{\log X, 1/\beta\} \cdot \frac{(\log X)^B}{T'} d\beta \ll \frac{1}{\log X} \left((\log X) \cdot \frac{1}{\log X } + \log\left(\eta \log X\right)\right) \ll \frac{\log\log X}{\log X}.
$$
Thus, suppose the maximum occurs with $|t| \leq T'$. Applying \cite[Lem. 3.1]{GHS}, we get
\begin{align*}
\max_{|t| \leq T'} |\mc{F}(c_0 + \beta + it)(1-w^{-\beta-it})| &\leq \max_{|t| \leq (\log X)^B} |\mc{F}(c_0 + it)(1-w^{-it})| + O\left(\frac{\beta}{(\log X)^B} \sum_{n \leq X} \frac{|f(n)|}{n^{1+1/\log X}}\right) \\
&= \max_{|t| \leq (\log X)^B} |\mc{F}(c_0 + it)(1-w^{-it})| + O_{B,C}\left(\frac{\beta}{(\log X)^{B-1}} \mc{P}_f(X)\right).
\end{align*}
Inserting this into the $\beta$ integral yields, in this case,
\begin{align*}
T_1 &\ll_B \max_{|t| \leq (\log X)^B} |\mc{F}(c_0+it)(1-w^{-it})| \cdot \frac{1}{\log X} \int_0^{\eta} \min\{\log X, 1/\beta\} d\beta + \frac{\mc{P}_f(X)}{(\log X)^B}\int_0^{\eta} \beta \min\{\log X,1/\beta\} d\beta \\
&\ll_B \max_{|t| \leq (\log X)^B} |\mc{F}(1+1/\log X+it)(1-w^{-it})| \cdot \frac{\log(\log X/\log (ew))}{\log X} + \frac{\mc{P}_f(X)}{(\log X)^B}. 
\end{align*}
Finally, we focus on the maximum here. Note that $|1-w^{-it}| \ll \min\{1,|t|\log (ew)\}$, so combining Lemma \ref{lem:LfnBd} with our hypothesis \eqref{eq:hyp4}, we obtain
\begin{align*}
&\max_{|t| \leq (\log X)^B} |\mc{F}(1+1/\log X + it)(1-w^{-it})| \\
&\ll_{A,B,C} (\log X) \mc{P}_f(X)  \cdot \max_{|t| \leq (\log X)^B}\min\{1,|t| \log (ew)\} e^{-\rho(f,n^{i(t_0+t)};X)^2} \\
&\ll_{A,B} (\log X) \mc{P}_f(X) \max_{|t| \leq (\log X)^B} \min\{1,|t|\log (ew)\} \left(\frac{1}{(\log X)^{\sg}} + \frac{1}{(1+|t|\log X)^{\sg}}\right) \\
&\ll (\log X) \mc{P}_f(X) \cdot \left(\frac{\log (ew)}{\log X}\right)^{\min\{1,\sg\}}.
\end{align*}
Hence, as $\hat{\sg} = \min\{1,\sg\} \leq A \leq B$, and $(\log X)^{-1} \ll_A (\log X)^{-A} \mc{P}_f(X)$ by \eqref{eq:hyp3'} we get
\begin{align*}
T_1 &\ll_{A,B,C} \log\left(\frac{\log X}{\log(ew)}\right) \left(\frac{\log (ew)}{\log X}\right)^{\hat{\sg}} \mc{P}_f(X)  + \frac{\mc{P}_f(X)}{(\log X)^B} + \frac{\log\log X}{\log X} \\
&\ll \log\left(\frac{\log X}{\log(ew)}\right) \left(\frac{\log (ew) + \log\log X}{\log X}\right)^{\hat{\sg}} \mc{P}_f(X).
\end{align*}
Combining all of these bounds and inserting them into \eqref{eq:intstoBd}, we thus find that
\begin{align*}
&\left|\frac{1}{X}\sum_{n \leq X} f(n)n^{-it_0} - \frac{w}{X}\sum_{n \leq X/w} f(n)n^{-it_0}\right| \\
&\ll_{A,B,C} \mc{P}_f(X) \log\left(\frac{\log X}{\log(ew)}\right)\left(\left(\frac{\log (ew) + \log\log X}{\log X}\right)^{\min\{1,A\}}+ \left(\frac{\log (ew) + \log\log X}{\log X}\right)^{\hat{\sg}}\right) \\
&\ll \mc{P}_f(X) \log\left(\frac{\log X}{\log(ew)}\right) \left(\frac{\log(ew) + \log\log X}{\log X}\right)^{\hat{\sg}}. 
\end{align*}
This proves the first claim. \\
The second claim can be deduced similarly, since (using the same notation as above) in the first step we have (after shifting $s \mapsto s-it_0$)
\begin{align*}
&\frac{1}{X^{1+it_0}}\sum_{n \leq X} f(n) - \left(\frac{w}{X}\right)^{1+it_0}\sum_{n \leq X/w} f(n) \\
&= \int_0^{\eta}\int_0^{\eta} \frac{1}{\pi i} \int_{c_0-i\infty}^{c_0+i\infty} \mc{S}(s-\alpha-\beta) \mc{L}(s+\beta) \frac{\mc{L}'}{\mc{L}}(s+\alpha)\frac{\mc{L}'}{\mc{L}}(s+\alpha+2\beta) \frac{X^{s-1} (1-w^{1-s})}{s+it_0} ds d\beta d\alpha \\
&+ O_{A,B,C}\left(\left(\frac{\log y}{\log X}\right)^A\mc{P}_f(X)\right),
\end{align*}
which simply localizes the argument above to the range $|t+t_0| \leq T$ instead.
\end{proof}

Theorem \ref{thm:Lipschitz} may be directly applied to obtain the first estimate in Theorem \ref{thm:compLongSums}.
To obtain the second, we will use the following corollary of Theorem \ref{thm:Lipschitz} that allows us to pass from $n^{-it_0}$-twisted sums to untwisted sums of $f(n)$ on long intervals.

\begin{cor}\label{cor:tShift}
Let $t_0 = t_0(f,X)$ be as above. Then for any $x \in (X/2,X]$,
$$
\frac{1}{x}\sum_{n \leq x} f(n)n^{-it_0} = \frac{1+it_0}{x^{1+it_0}} \sum_{n \leq x} f(n) + O_{A,B,C}\left(|t_0| \mc{P}_f(X) \frac{(\log\log X)^{\hat{\sg}+1}}{(\log X)^{\hat{\sg}}}\right).
$$
\end{cor}
\begin{proof}
By partial summation, we have
\begin{equation}\label{eq:PSLipschitz}
\frac{1}{x}\sum_{n \leq x} f(n)n^{-it_0} = \frac{1}{x}\int_1^x u^{-it_0} d\{\sum_{n \leq x} f(n)\} = \frac{1}{x^{1+it_0}} \sum_{n \leq x} f(n) +\frac{it_0}{x} \int_1^x \frac{1}{u^{1+it_0}}\sum_{n \leq u} f(n) du.
\end{equation}
We split the integral over $u$ at $x/(\log X)$. In the first range we use the trivial bound together with Lemma \ref{lem:Shiu}, obtaining
$$
\leq \frac{|t_0|}{x} \int_1^{x/(\log X)^2} \left(\frac{1}{u} \sum_{n \leq u} |f(n)|\right) du \ll_{B,C} \frac{|t_0|}{(\log X)^2} \prod_{p \leq X} \left(1+\frac{|f(p)|}{p}\right) \ll \frac{|t_0|}{\log X}\mc{P}_f(X).
$$
In the remaining range $x/(\log X)^2 < u \leq x$ we apply the second claim in Theorem \ref{thm:Lipschitz} (with $1 \leq w \leq (\log X)^2$), which gives
\begin{align*}
&\frac{it_0}{x} \int_{x/(\log X)^2}^X \left(\frac{1}{x^{1+it_0}} \sum_{n \leq x} f(n) + O_{A,B,C} \left(\mc{P}_f(X) \frac{(\log\log X)^{\hat{\sg}+1}}{(\log X)^{\hat{\sg}}}\right)\right) du \\
&= \left(\frac{1}{x^{1+it_0}}\sum_{n \leq x} f(n) \right) \cdot \frac{it_0}{x} \int_{x/(\log X)^2}^x du + O_{A,B,C} \left(|t_0| \mc{P}_f(X) \frac{(\log\log X)^{\hat{\sg} + 1}}{(\log X)^{\hat{\sg}}}\right) \\
&= \frac{it_0}{x^{1+it_0}} \sum_{n \leq x} f(n) + O_{A,B,C} \left(|t_0| \mc{P}_f(X) \frac{(\log\log X)^{\hat{\sg} + 1}}{(\log X)^{\hat{\sg}}}\right).
\end{align*}
Combining this with the estimate from $1 \leq u \leq x/(\log X)^2$, then inserting this into \eqref{eq:PSLipschitz}, we prove the claim.
\end{proof}

\begin{proof}[Proof of Theorem \ref{thm:compLongSums}]
We begin by proving the first estimate in the statement of the theorem.
Let $X/(\log X)^{\hat{\sg}/2} \leq h \leq X/10$. Note that by writing $x-h = x/w$, where $w := (1-h/x)^{-1} \in [1,2]$, the LHS is
\begin{equation}\label{eq:hsumf}
\frac{1}{h} \left( \sum_{n \leq x} f(n)n^{-it_0} -\sum_{n \leq x/w} f(n)n^{-it_0}\right).
\end{equation}
and similarly the main term in the RHS is
\begin{equation}\label{eq:Xsumf}
\frac{2}{X}\left(\sum_{n \leq X} f(n)n^{-it_0}-\sum_{n \leq X/2}f(n)n^{-it_0}\right).
\end{equation}
By Theorem \ref{thm:Lipschitz}, \eqref{eq:hsumf} becomes
\begin{align*}
&\frac{1}{h}\left( \left(1-\frac{1}{w}\right) \sum_{n \leq x} f(n)n^{-it_0} \right) + O_{A,B,C} \left(\frac{X}{h} \frac{(\log\log X)^{\hat{\sg}+1}}{(\log X)^{\hat{\sg}}}\mc{P}_f(X)\right) \\
&= \frac{1}{x}\sum_{n \leq x} f(n)n^{-it_0} + O_{A,B,C} \left(\frac{(\log\log X)^{\hat{\sg}+1}}{(\log X)^{\hat{\sg}/2}} \mc{P}_f(X)\right).
\end{align*}
%
%
Similarly, applying Theorem \ref{thm:Lipschitz} twice to \eqref{eq:Xsumf}, we also find that
\begin{align}
\frac{2}{X}\sum_{X/2 < n \leq X} f(n)n^{-it_0} &= \frac{2}{X}\left(1-\frac{1}{2}\right) \sum_{n \leq X} f(n)n^{-it_0} + O_{A,B,C} \left(\frac{(\log\log X)^{\hat{\sg}+1}}{(\log X)^{\hat{\sg}}}\mc{P}_f(X)\right) \nonumber\\
&= \frac{1}{x} \sum_{n \leq x} f(n)n^{-it_0} + O_{A,B,C}\left(\frac{(\log\log X)^{\hat{\sg}+1}}{(\log X)^{\hat{\sg}}}\mc{P}_f(X)\right), \label{eq:xtoX}
\end{align}
viewing $x = X/u$ for some $u \in [1,2]$ in the last step. Combined with the previous estimate, we deduce the first claimed estimate of the theorem.

To prove the second claimed estimate we apply Corollary \ref{cor:tShift} to obtain
\begin{align*}
\sum_{x-h < n \leq x} f(n) = \frac{x^{it_0}}{1+it_0} \sum_{n \leq x} f(n)n^{-it_0} - \frac{(x-h)^{it_0}}{1+it_0} \sum_{n \leq x-h} f(n)n^{-it_0} + O\left(X \mc{P}_f(X) \frac{(\log\log X)^{\hat{\sg} + 1}}{(\log X)^{\hat{\sg}}}\right).
\end{align*}
Setting $w := (1-h/x)^{-1}$ as in \eqref{eq:hsumf}, we have
\begin{align*}
&\frac{x^{it_0}}{1+it_0} \sum_{n \leq x} f(n)n^{-it_0} - \frac{(x-h)^{it_0}}{1+it_0} \sum_{n \leq x-h} f(n)n^{-it_0} \\
&= \left(\frac{x^{1+it_0}}{1+it_0} - \frac{1}{w} \frac{(x-h)^{it_0}}{1+it_0}\right) \sum_{n \leq x} f(n)n^{-it_0} + O\left(X \mc{P}_f(X) \frac{(\log\log X)^{\hat{\sg} + 1}}{(\log X)^{\hat{\sg}}}\right) \\
&= \frac{x^{it_0} - (x-h)^{1+it_0}}{1+it_0} \cdot \frac{1}{x} \sum_{n \leq x} f(n)n^{-it_0} + O\left(X \mc{P}_f(X) \frac{(\log\log X)^{\hat{\sg} + 1}}{(\log X)^{\hat{\sg}}}\right).
\end{align*}
By \eqref{eq:xtoX}, the main term here is 
$$
\int_{x-h}^x u^{it_0} du \cdot \frac{2}{X}\sum_{X/2 < n \leq X} f(n)n^{-it_0} + O\left(h \mc{P}_f(X) \frac{(\log\log X)^{\hat{\sg} + 1}}{(\log X)^{\hat{\sg}}}\right).
$$
Combining these estimates, we thus obtain
$$
\frac{1}{h} \sum_{x-h < n \leq x} f(n) = \frac{1}{h}\int_{x-h}^x u^{it_0} du \cdot \frac{2}{X} \sum_{X/2 < n \leq x} f(n)n^{-it_0} + O\left(\frac{X}{h} \mc{P}_f(X)\frac{(\log\log X)^{\hat{\sg} + 1}}{(\log X)^{\hat{\sg}}}\right),
$$
and so the second claimed estimate of the theorem follows from $h\geq X/(\log X)^{\hat{\sg}/2}$.
\end{proof}

\section{Applying the Matom\"{a}ki-Radziwi\l\l{} Method}
In this section, which broadly follows the lines of the proof of \cite[Thm. 3]{MR}, we set out the key elements of the proof of Theorem \ref{thm:MRDB}.
\subsection{Large sieve estimates}
The content of this subsection can essentially all be found in \cite[Sec. 6 and 7]{MR} and in \cite[Sec. 3]{MRII}. 
In what follows, a set $\mc{T} \subset \mb{R}$ is said to be \emph{well-spaced} if $|t_1-t_2| \geq 1$ for any distinct $t_1,t_2 \in \mc{T}$.
\begin{lem}[Sparse large sieve for multiplicative sequences] \label{lem:LSInts}
Let $T \geq 1$ and $2 \leq N \leq X$. Let $\{a_n\}_{n \leq N}$ be a sequence of complex numbers. Let $\mc{T} \subset [-T,T]$ be well-spaced.  The following bounds hold:
\begin{enumerate}[(a)]
\item ($L^2$ mean value theorem, sparse version)
\begin{align*}
\int_{-T}^T |\sum_{n \leq N} a_n n^{-it}|^2 dt &\ll T\sum_{n \leq N} |a_n|^2 + T\sum_{n \leq N} \sum_{1 \leq |m| \leq n/T} |a_na_{m+n}|.
\end{align*}
\item ($L^2$ mean value theorem with multiplicative majorant) Let  $1 \leq M \leq N$, and let $c > 0$. Assume there is a multiplicative function $f: \mb{N} \ra \mb{C}$ satisfying $|f(n)| \leq d_B(n)^C$ such that $|a_n| \leq c|f(n)|$ for all $n\leq N$. Then 
\begin{align*}
\int_{-T}^T |\sum_{N-M < n \leq N} \frac{a_nn^{-it}}{n}|^2 
&\ll_{B,C} c^2\left(\frac{TM}{N^2} \mc{P}_{f^2}(N) + \frac{M}{N}\mc{P}_f(N)^2\right).
\end{align*}
\item (Discrete mean value theorem) 
$$
\sum_{t \in \mc{T}} |\sum_{N/3 < n \leq N} \frac{a_nn^{-it}}{n}|^2 \ll \min\left\{\left(1+\frac{T}{N}\right) \log(2N) , \left(1+|\mc{T}|\frac{T^{1/2}}{N}\right)\log(2T)\right\} \frac{1}{N} \sum_{N/3 < n \leq N} |a_n|^2.
$$
\end{enumerate}
\end{lem}
\begin{proof}
Part (a) is \cite[Lem. 3.2]{MRII}, part (b) is proven in the same way as\footnote{The technique used there relies on the main result of \cite{Hen}, which is valid generally for multiplicative functions that are bounded by a power of the divisor function.} \cite[Lem. 3.4]{MRII} and part (c) is a combination of \cite[Lem. 7 and 9]{MR}.
\end{proof}

\begin{lem}[Large Sieve with Prime Support] \label{lem:LSPrim}
Let $B,T \geq 1$, $P \geq 10$. Let $\{a_p\}_{P < p \leq 2P}$ be a sequence with $\max_{P < p \leq 2P}|a_p| \leq B$. Let $P(s) := \sum_{P < p \leq 2P} a_pp^{-s}$, for $s \in \mb{C}$ and let $\mc{T} \subset [-T,T]$ be a well-spaced set. 
\begin{enumerate}[(a)]
\item (Hal\'{a}sz-Montgomery estimate for primes)
$$
\sum_{t \in \mc{T}} |P(1+it)|^2 \ll_B \frac{1}{(\log P)^2} \left(1+ |\mc{T}|(\log T)^2 \exp\left(-\frac{\log P}{(\log T)^{2/3+\e}}\right)\right).
$$
\item (Large values estimate) If $\mc{T}$ consists only of $t \in [-T,T]$ with $|P(1+it)| \geq V^{-1}$ then
$$
|\mc{T}| \ll_B T^{2\frac{\log V}{\log P}} V^2 \exp\left(2B\frac{\log T}{\log P}\log\log T\right).
$$
\end{enumerate}
\end{lem}
\begin{proof}
Part (a) is \cite[Lem. 11]{MR}, while part (b) is proven precisely as in \cite[Lem. 8]{MR}, keeping track of the upper bound condition $|a_p| \leq B$.
\end{proof}

\subsection{Dirichlet Polynomial Decomposition}
The following is a variant of \cite[Lem. 12]{MR} tailored to elements of $\mc{M}(X;A,B,C;\gamma,\sg)$. 
\begin{lem}\label{lem:decomp}
Let $f \in \mc{M}(X;A,B,C;\gamma,\sg)$. Let $T \geq 1$, $1 \leq H \leq X^{1/2}$ and $1 \leq P \leq Q \leq X^{1/2}$. Set $\mc{I}$ to be the interval of integers $\llf H\log P \rrf \leq v \leq H \log Q$, and let $\mc{T} \subset [-T,T]$. Then
\begin{align*}
\int_{\mc{T}} &\left|\sum_{\substack{X/3 < n \leq X \\ \omega_{[P,Q]}(n) \geq 1}} \frac{f(n)}{n^{1+it}}\right|^2 dt \ll_{B,C}  H \log(Q/P) \sum_{v \in \mc{I}} \int_{\mc{T}} |Q_{v,H}(1+it)R_{v,H}(1+it)|^2 dt \\
&+ \left(\frac{1}{H}+\frac{1}{P}\right) \left(\frac{T}{X}\mc{P}_{f^2}(X) + \mc{P}_f(X)^2\right),
\end{align*}
where for $v \in \mc{I}$ and $s \in \mb{C}$ we have set
\begin{align*}
Q_{v,H}(s) &:= \sum_{\substack{P \leq p \leq Q \\ v/H \leq \log p \leq (v+1)/H}} f(p)p^{-s} \\
R_{v,H}(s) &:= \sum_{Xe^{-v/H}/3 \leq m \leq Xe^{-v/H}} \frac{f(m)}{m^s(\omega_{[P,Q]}(m) + 1)}.
\end{align*}
\end{lem}
\begin{proof}
The proof is the same as that of \cite[Lem. 12]{MR} (with $a_n = f(n)1_{\omega_{[P,Q]} \geq 1}(n)$, $b_m = f(m)$ and $c_p = f(p)$ for $P \leq p \leq Q$), with appropriate appeal to Lemma \ref{lem:LSInts} in place of the usual mean value theorem. For example (as on \cite[top of p.20]{MR}), for $Y \in \{X/(3Q),X/P\}$ we have
\begin{align*}
\int_{\mc{T}}\left|\sum_{m \in [Ye^{-1/H},Ye^{1/H}]} f(m)m^{-1-it}\right|^2 dt
 &\ll_{B,C} \frac{T}{XH} \mc{P}_{f^2}(X) + \frac{1}{H}\mc{P}_f(X)^2.
 \end{align*}
 \end{proof}
 
 \begin{lem} \label{lem:MixedMom}
Let $f \in \mc{M}(X;A,B,C;\gamma,\sg)$. Let $T,Y_1,Y_2 \geq 1$, $1 \leq X' \leq X/Y_1$ and let $\ell := \lceil \frac{\log Y_2}{\log Y_1}\rceil$. Define
 $$
 Q(s) := \sum_{Y_1/2 < p \leq Y_1} c_pp^{-s}, \quad \quad A(s) := \sum_{X'/(2Y_2) < m \leq X'/Y_2} f(m)m^{-s},
 $$
 where $|c_p| \leq B$ for all $Y_1/2<p \leq Y_1$. Finally, let $\mc{T} \subseteq [-T,T]$. Then
 $$
 \int_{\mc{T}} |Q(1+it)^{\ell}A(1+it)|^2 dt \ll_{B,C} B^{2\ell} (\ell!)^2 \left(\frac{T}{X} \mc{P}_{f^2}(X) + \mc{P}_f(X)^2\right).
 $$
 \end{lem}
 \begin{proof}
Writing $c_p := B c_p'$, where now $|c_p'| \leq 1$ and letting $\tilde{Q}(s)$ denote the Dirichlet polynomial with $c_p$ replaced by $c_p'$ for all $p \in (Y_1/2,Y_1]$, the LHS in the statement is
\[
\ll B^{2\ell} \int_{\mc{T}} |\tilde{Q}(1+it) A(1+it)|^2 dt.
\]
The rest of the proof is essentially the same as that of \cite[Lem. 7.1]{MRII}, save that the function $g^{\ast}$ is replaced by $|f|$ (this does not affect the proof, which depends on our Lemma \ref{lem:Shiu} and \cite[Thm. 3]{Hen}, both of which also apply to $|f|$.)
%
\end{proof}

\subsection{Integral Averages of Dirichlet Polynomials}
As above, we write $t_0 = t_0(f,X)$ to denote an element of $[-X,X]$ that minimizes the map
$$
t \mapsto \sum_{p \leq X} \frac{|f(p)| - \text{Re}(f(p)p^{-it})}{p} = \rho(f, n^{it};X)^2.
$$
\begin{prop} \label{prop:Pars}
Let $1/\log X \leq \delta \leq 1$, and put $I_{\delta}:= [t_0-\delta^{-1},t_0 + \delta^{-1}]$. Let $\{a_n\}_{n \leq X}$ be a sequence of complex numbers. Then
\begin{align*}
&\frac{2}{X}\int_{X/2}^X \left|\frac{1}{h}\sum_{\substack{x-h < m \leq x}} a_m - \frac{1}{2\pi} \int_{I_{\delta}} A(1+it) \frac{x^{1+it}-(x-h)^{1+it}}{h(1+it)}dt\right|^2dx \\
&\ll \int_{[-X/h, X/h] \bk I_{\delta}} |A(1+it)|^2 dt + \max_{T \geq X/h} \frac{X}{hT}\int_T^{2T} |A(1+it)|^2 dt,
\end{align*}
where we have set
$$
A(s) := \sum_{\substack{X/3 < n \leq X }} \frac{a(n)}{n^s}, \quad\quad s \in \mb{C}.
$$
\end{prop}
\begin{proof}
For each $x \in [X/2,X]$, Perron's formula gives
$$
\frac{1}{h}\sum_{x-h < m \leq x} a_m = \frac{1}{2\pi} \int_{\mb{R}} A(1+it) \frac{x^{1+it}-(x-h)^{1+it}}{h(1+it)} dt.
$$
Subtracting the contribution from $|t-t_0| \leq \delta^{-1}$, the LHS in the statement becomes
$$
\frac{2}{X}\int_{X/2}^X\left|\frac{1}{2\pi} \int_{\mb{R} \bk I_{\delta}} A(1+it) \frac{x^{1+it}-(x-h)^{1+it}}{h(1+it)} dt\right|^2 dx.
$$
The remainder of the proof is identical to that of \cite[Lemma 14]{MR} (which only uses the boundedness of the coefficients $\{a_n\}_n$ there for the corresponding contribution of $I_{\delta}$).
\end{proof}

\subsection{Restricting to a ``nicely factored'' set}
We fix parameters $\eta \in (0,1/12)$, $Q_1 := h_0$, $P_1 := (\log h_0)^{40B/\eta}$ and $P_j := \exp\left(j^{4j-2} (\log Q_1)^{j-1}(\log P_1)\right)$, $Q_j := \exp\left(j^{4j} (\log Q_1)^j\right)$, for $1 \leq j \leq J$, where $J$ is maximal with $Q_J \leq \exp\left(\sqrt{\log X}\right)$. We highlight the different choice of $P_1$, with all other choices being the same as in \cite[Sec. 2 and 8]{MR}. We also let
$$
\mc{S} = \mc{S}_{X,P_1,Q_1} := \{n \leq X : \omega_{[P_j,Q_j]}(n) \geq 1 \text{ for all } 1 \leq j \leq J\}.
$$
\begin{rem} \label{rem:params}
The following properties may be verified directly, as long as $h_0$ is sufficiently large (in terms of $B$):
\begin{enumerate}
\item $\log P_j \geq 
\frac{8Bj^2}{\eta} \log\log (2BQ_{j+1}) \text{ for all $1 \leq j \leq J$}$
\item $\log Q_j \leq 2^{4j} (\log Q_{j-1})(\log Q_1) \leq (\log Q_{j-1})^{3} \leq Q_{j-1}^{1/24}$
\item $\frac{\log P_j}{\log Q_j} = \frac{\log P_1}{j^2 \log Q_1}$ for $2 \leq j \leq J$, so the terms $\{\log P_j/\log Q_j\}_{j \geq 1}$ are summable.
\end{enumerate}
We will use these in due course.
\end{rem}

We wish to reduce our work to handling short and long averages with $n$ restricted to the set $\mc{S}$. To handle averages of $f(n)$ for $n \notin \mc{S}$ we use the following result. For a set of integers $\mc{A}$, we write $(n,\mc{A}) = 1$ to mean that $(n,a) = 1$ for all $a \in \mc{A}$.

\begin{lem} \label{lem:contS}
Let $1 < P \leq Q \leq X$, and let $f \in \mc{M}(X;A,B,C;\gamma,\sg)$. Then for any $10 \leq h_0 \leq X/10H(f;X)$ and $h := h_0H(f;X)$,
$$
\frac{2}{X}\int_{X/2}^{X} \left|\frac{1}{h}\sum_{\substack{x-h < m \leq x \\ (m,[P,Q]) = 1}} f(m)\right|^2 dx \ll_{A,B,C} \left(\left(\frac{\log P}{\log Q} \right)^A+ \frac{1}{h_0}\right) \mc{P}_f(X)^2.
$$
In particular,
$$
\frac{2}{X} \int_{X/2}^{X} \left|\frac{1}{h} \sum_{\substack{x-h < m \leq x \\ m \notin \mc{S}}}  f(m)\right|^2 dx \ll_{A,B,C} \left(\frac{\log \log h_0}{\log h_0}\right)^{A} \mc{P}_f(X)^2.
$$
\end{lem}
\begin{proof}
Expanding the square and applying the triangle inequality, the LHS is
\begin{align*}
&\leq \frac{2}{h^2X} \sum_{\ss{X/2-h < m_1,m_2 \leq X \\ |m_1-m_2| \leq h \\ (m_1m_2,[P,Q]) = 1}} |f(m_1)\bar{f(m_2)}| \int_{X/2}^X 1_{[m_1,m_1+h)}(x) 1_{[m_2,m_2+h)}(x) dx  \\
&\ll \frac{1}{hX} \sum_{\ss{X/3 < m \leq X}} |f(m)|^2 + \frac{1}{hX} \sum_{1 \leq |l| \leq h} \sum_{\ss{X/3 < m \leq X \\ (m,[P,Q]) = 1}} |f(m)||f(m+l)|.
\end{align*}
By Lemma \ref{lem:Shiu} and \eqref{eq:HfXBd}, the first term on the RHS is bounded as
\begin{equation}\label{eq:DiagBd}
\ll_{B,C} \frac{1}{h_0H(f;X)}\mc{P}_{f^2}(X) \ll_B \frac{1}{h_0}\mc{P}_f(X)^2.
\end{equation}
Next, to bound the correlation sums we apply\footnote{This result is stated in \cite{MRII} for bounded multiplicative functions, but the proof there works identically for divisor-bounded functions as well since it relies principally on the general setup of \cite[Thm. 3]{Hen}.} \cite[Lem. 3.3]{MRII} (with $r_1 = r_2 = 1$) to the pair of multiplicative functions $f1_{\mc{S}^c}$ and $f$, which gives 
$$
\sum_{1 \leq |l| \leq h} \sum_{\ss{X/3 < n \leq X \\ (n,[P,Q]) = 1}} |f(n)f(n+l)| \ll_{B,C} hX\mc{P}_{f1_{(m,[P,Q]) = 1}}(X)\mc{P}_f(X).
$$
Since $f \in \mc{M}(X;A,B,C;\gamma,\sg)$, by \eqref{eq:hyp3'} we get
\begin{align*}
\mc{P}_{f1_{(m,[P,Q]) = 1}}(X) \asymp_B \mc{P}_f(X) \prod_{P \leq p \leq Q}\left(1+\frac{1-|f(p)|}{p}\right)\left(1-\frac{1}{p}\right) &\ll_{A,B} \mc{P}_f(X) \exp\left(A \frac{\log P}{\log Q}\right) \\
&= \left(\frac{\log P}{\log Q}\right)^{A} \mc{P}_f(X).
\end{align*}
It follows that
$$
\frac{1}{hX}\sum_{1 \leq |l| \leq h} \sum_{\ss{X/3 < n \leq X \\ n \notin \mc{S}}} |f(n)||f(n+l)| \ll_{A,B} \left(\frac{\log P}{\log Q}\right)^{A} \mc{P}_f(X)^2.
$$
The first claim now follows upon combining this with \eqref{eq:DiagBd}. 
The second claim follows similarly, save that in the argument above the term $\mc{P}_{f1_{(m,[P,Q]) = 1}}(X)$ is replaced by
\begin{align*}
\mc{P}_{f1_{\mc{S}^c}}(X) &\ll_B \mc{P}_f(X)\prod_{1 \leq j \leq J} \prod_{P_j \leq p \leq Q_j} \left(1+\frac{1-|f(p)|}{p}\right)\left(1-\frac{1}{p}\right) \ll_{A,B} \mc{P}_f(X) \exp\left(A \sum_{j \geq 1} \frac{\log P_j}{\log Q_j}\right) \\
&\ll \left(\frac{\log P_1}{\log Q_1}\right)^{A} \mc{P}_f(X) \ll_{A,B} \left(\frac{\log\log h_0}{\log h_0}\right)^A \mc{P}_f(X).
\end{align*}
\end{proof}

Having disposed of $n \notin \mc{S}$, we now concerntrate on $n \in \mc{S}$. To prove Theorem \ref{thm:MRDB} we will apply Proposition \ref{prop:Pars} to the sequence $a_m = f(m)1_{m \in \mc{S}}$, in combination with the following key proposition. \\
Recall that $\kappa := \frac{\min\{1,\sg\}}{8B + 21}$. We also define $\Delta := (2B+5)\kappa$, and 
$$
F(s) := \sum_{\ss{ X/3 < n \leq X \\ n \in \mc{S}}} \frac{f(n)}{n^s}, s \in \mb{C}.
$$
\begin{prop} \label{prop:L2DPInt}
Set $\delta = (\log X)^{-\Delta}$. Then 
$$
\int_{[-X/h,X/h] \bk I_{\delta}} |F(1+it)|^2 dt \ll_{A,B,C} \left(\frac{(\log Q_1)^{1/3}}{P_1^{1/6-2\eta}} + \left(\frac{\log\log X}{(\log X)^{\kappa}}\right)^{\min\{1,A\}} \right) \mc{P}_f(X)^2.
$$
\end{prop}
\begin{rem}\label{rem:2ndInt}
We remark that both terms in Proposition \ref{prop:Pars} can be treated using Proposition \ref{prop:L2DPInt}. Indeed, by Lemma \ref{lem:LSInts}(b),
$$
\frac{1}{T}\int_T^{2T} |F(1+it)|^2 dt \ll_{B,C} \frac{1}{X} \mc{P}_{f^2}(X) + \frac{1}{T}\mc{P}_f(X)^2.
$$
and therefore
$$
\max_{T \geq X(\log h_0)^A/h} \frac{X/h}{T} \int_T^{2T} |F(1+it)|^2 dt \ll_{B,C} \frac{1}{h}\mc{P}_{f^2}(X) +  \frac{1}{(\log h_0)^A} \mc{P}_{f}(X)^2 \ll \frac{1}{(\log h_0)^A} \mc{P}_f(X)^2,
$$
which is obviously sufficient in Theorem \ref{thm:MRDB}. We clearly also have
$$
\max_{X/h < T \leq X(\log h_0)^A/h} \frac{X/h}{T} \int_T^{2T} |F(1+it)|^2 dt \leq \int_{X/h}^{X(\log h_0)^A/h} |F(1+it)|^2 dt.
$$
This expression will also be bounded using Proposition \ref{prop:L2DPInt} with $h$ replaced by $h/(\log h_0)^A$, which does not change the form of the final estimates. 
\end{rem}
\subsection{Proof of Proposition \ref{prop:L2DPInt}}
The proof follows the same lines as those in \cite[Sec. 8]{MR}, save that we apply our versions of the corresponding lemmas  that address the growth of $f \in \mc{M}(X;A,B,C;\gamma,\sg)$. We sketch the details here, emphasizing the main differences. \\
We select parameters $\alpha_1,\ldots,\alpha_J$ such that $\alpha_j := \frac{1}{4}-\eta \left(1+\frac{1}{2j}\right)$.
For each $1 \leq j \leq J$ let 
$$
\mc{I}_j := [\llf H_j \log P_j \rrf, H_j\log Q_j] \cap \mb{Z}, \quad \quad H_j := j^2 \frac{P_1^{1/6-\eta}}{(\log Q_1)^{1/3}}.
$$
Also, with $s \in \mb{C}$ we let
$$
Q_{v,H_j}(s) := \sum_{\substack{P_j \leq p \leq Q_j \\ e^{v/H_j} \leq p \leq e^{(v+1)/H_j}}} f(p)p^{-s}, \quad\quad  R_{v,H_j}(s) := \sum_{\frac{1}{3}Xe^{-v/H_j} < Xe^{-v/H_j}} \frac{f(m)}{m^s(1+\omega_{[P_j,Q_j]}(m))}.
$$
We split the set of $t \in \mc{X} := [-X/h,X/h] \bk I_{\delta}$ into sets
\begin{align*}
\mc{T}_1 &:= \{t \in \mc{X} : |Q_{v,H_1}(1+it)| \leq e^{-\alpha_1v/H_1} \text{ for all } v_1 \in \mc{I}_1\} \\
\mc{T}_j &:= \{t \in \mc{X} : |Q_{v,H_j}(1+it)| \leq e^{-\alpha_jv/H_j} \text{ for all } v_j \in \mc{I}_j\} \bk \bigcup_{1 \leq i \leq j-1} \mc{T}_i, \quad 2 \leq j \leq J, \text{ if } J \geq 2 \\
\mc{U} &:= \mc{X} \bk \bigcup_{1 \leq j \leq J} \mc{T}_j.
\end{align*}
We thus have
$$
\int_{[-X/h,X/h] \bk I_{\delta}} |F(1+it)|^2 dt = \sum_{1 \leq j \leq J} \int_{\mc{T}_j} |F(1+it)|^2 dt + \int_{\mc{U}} |F(1+it)|^2 dt.
$$
We estimate the contributions from $\mc{T}_j$ as in \cite{MR}, save that in the applications of the large sieve inequalities we use Lemma \ref{lem:LSInts}(b); as an example we will give full details for $j = 1$ and highlight the main changes for the corresponding bounds for $2 \leq j \leq J$. \\
In each case we apply Lemma \ref{lem:decomp} to obtain
$$
\int_{\mc{T}_j} |F(1+it)|^2 dt \ll_{B,C} \mc{M}_j + \mc{E}_j,
$$
where we set 
\begin{align*}
\mc{M}_j &:= H_j\log(Q_j/P_j) \sum_{v \in \mc{I}_j} \int_{\mc{T}_j} |Q_{v,H_j}(1+it)R_{v,H_j}(1+it)|^2 dt \\
\mc{E}_j &:= \left(\frac{1}{H_j}+\frac{1}{P_j}\right) \left(\frac{T}{X} \mc{P}_{f^2}(X) + \mc{P}_f(X)^2\right).
\end{align*}
The choice of parameters gives, by \eqref{eq:DiagBd},
\begin{equation}\label{eq:Pf2toPf}
\sum_{1 \leq j \leq J} \mc{E}_j \ll \frac{(\log Q_1)^{1/3}}{P_1^{1/6-\eta}}\left(\frac{1}{h_0H(f;X)} \mc{P}_{f^2}(X) + \mc{P}_f(X)^2\right) \ll_{B,C} \frac{(\log Q_1)^{1/3}}{P_1^{1/6-\eta}} \mc{P}_f(X)^2,
\end{equation}
since $\sum_j P_j^{-1} \ll P_1^{-1}$ and $H(f;X)^{-1} \mc{P}_{f^2}(X) \ll_{B,C} \mc{P}_f(X)^2$. \\
We next consider the contribution from the main terms. When $j = 1$, since $v/H_1 \leq \log Q_1 = \log h_0$, we have
\begin{align*}
\mc{M}_1 &\leq H_1\log Q_1\sum_{v \in \mc{I}_1} e^{-2\alpha_1 v/H_1} \int_{-X/h}^{X/h} |R_{v,H_j}(1+it)|^2 dt \\
&\ll_{B,C} H_1\log Q_1\sum_{v \in \mc{I}_1} e^{-2\alpha_1 v/H_1} \left(\mc{P}_f(X)^2+ \frac{e^{v/H_1}}{h_0H(f;X)} \mc{P}_{f^2}(X)\right)  \\
&\ll_{B,C} \frac{H_1^2\log Q_1}{P_1^{2\alpha_1}}\mc{P}_{f}(X)^2,
\end{align*}
treating $\mc{P}_{f^2}(X)$ using \eqref{eq:DiagBd}. Since $H_1^2\log Q_1 = P_1^{1/3-2\eta}(\log Q_1)^{1/3} \leq P_1^{1/3-\eta}$,  $P_1^{2\alpha_1} \geq P_1^{1/2-3\eta}$ and $\eta \in (0,1/12)$, we get
$$
\mc{M}_1 \ll P_1^{-1/6+2\eta} \mc{P}_{f}(X)^2 \leq P_1^{-1/6+2\eta}\mc{P}_f(X)^2.
$$
Let now $2 \leq j \leq J$, if $J \geq 2$. By definition, we have
$$
\mc{T}_j = \bigcup_{r \in \mc{I}_{j-1}} \mc{T}_{j,r},
$$
where $\mc{T}_{j,r}$ is the set of all $t \in \mc{T}_j$ such that $|Q_{r,H_{j-1}}(1+it)| > e^{-\alpha_{j-1} r/H_{j-1}}$. 
Pointwise bounding $|Q_{v,H_j}(1+it)|$ for each $v \in \mc{I}_j$ leads to
\begin{align*}
\mc{M}_j&\leq H_j \log Q_j\sum_{v \in \mc{I}_j} \sum_{r \in \mc{I}_{j-1}} e^{-2\alpha_j v/H_j} \int_{\mc{T}_j} |R_{v,H_j}(1+it)|^2 dt \\
&\leq (H_j \log Q_j)|\mc{I}_j||\mc{I}_{j-1}| e^{-2v_0 \alpha_j + 2\ell_j r_0 \alpha_{j-1}} \int_{-X/h}^{X/h} |Q_{r_0,H_{j-1}}(1+it)^{\ell} R_{v,H_j}(1+it)|^2 dt,
\end{align*}
where $(r_0,v_0) \in \mc{I}_{j-1} \times \mc{I}_j$ yield the maximal contribution among all such pairs, and $\ell_j := \lceil\frac{v/H_j}{r/H_{j-1}}\rceil$. 
Using Lemma \ref{lem:MixedMom}, we get
$$
\mc{M}_j \ll_{B,C} (H_j \log Q_j)^3 e^{-2v_0\alpha_j + 2\ell_j r_0 \alpha_{j-1}} \exp\left(2\ell \log(2B\ell)\right) \left(1 + \frac{1}{h_0}\right) \mc{P}_{f}(X)^2.
$$
Minor modifications to the estimates in \cite[Sec. 8.2]{MR} (with $h_0$ in place of $h$ there), selecting $h_0$ sufficiently large in terms of $B$, shows that 
$$
\mc{M}_j \ll_{B,C} \frac{1}{j^2P_1}\mc{P}_{f}(X)^2,
$$
whence it follows that 
$$
\sum_{2 \leq j \leq J} \mc{M}_j \ll P_1^{-1} \mc{P}_{f}(X)^2
$$ 
(the requirements of our parameters $P_j,Q_j$ and $\alpha_j$ summarized in Remark \ref{rem:params} are sufficient for this). \\
Finally, we consider $\mc{U}$.  Set $H := (\log X)^{\kappa}$, $P = \exp((\log X)^{1-\kappa})$ and $Q = \exp(\log X/\log\log X)$, put $\mc{I} := [\llf H\log P\rrf, H\log Q] \cap \mb{Z}$, and define $Q_{v,H}$ and $R_{v,H}$ by
$$
Q_{v,H}(s) := \sum_{\substack{P \leq p \leq Q  \\ e^{v/H} \leq p \leq e^{(v+1)/H}}} f(p)p^{-s}, \quad\quad  R_{v,H}(s) := \sum_{\frac{1}{3}Xe^{-v/H} < Xe^{-v/H}} \frac{f(m)}{m^s(1+\omega_{[P,Q]}(m))}.
$$
Combining Lemma \ref{lem:decomp} with Lemma \ref{lem:LSInts} a) and the proof of Lemma \ref{lem:contS} to control those $n$ coprime to $P \leq p \leq Q$, we get that there is some $v_0 \in \mc{I}$ such that
\begin{align} \label{eq:UInt}
&\int_{\mc{U}} |F(1+it)|^2 dt \nonumber\\
&\ll_{B,C} (H \log X)^2 \int_{\mc{U}} |Q_{v_0,H}(1+it)R_{v_0,H}(1+it)|^2 dt + \int_{-T}^T \left|\sum_{\ss{X/3 < n \leq X \\ (n,[P,Q]) = 1}} \frac{f(n)}{n^{1+it}}\right|^2 dt + \mc{P}_f(X)^2\left(1+\frac{1}{h_0}\right)\left(\frac{1}{H} + \frac{1}{P}\right) \nonumber\\
&\ll_{B,C} (H \log X)^2 \int_{\mc{U}} |Q_{v_0,H}(1+it)R_{v_0,H}(1+it)|^2 dt + \mc{P}_f(X)^2\left(\frac{1}{H} + \frac{1}{P}+ \left(\frac{\log P}{\log Q}\right)^A\right) \nonumber\\
&\ll_{B,C} (H \log X)^2 \int_{\mc{U}} |Q_{v_0,H}(1+it)R_{v_0,H}(1+it)|^2 dt + \left(\frac{\log\log X}{(\log X)^{\kappa}}\right)^{\min\{1,A\}}\mc{P}_f(X)^2,
\end{align}
As in \cite[Sec. 8.3]{MR} we may select a discrete subset $\mc{V} \subset \mc{U}$ that is well-spaced, such that 
$$
\int_{\mc{U}} |Q_{v_0,H}(1+it)R_{v_0,H}(1+it)|^2 dt \ll \sum_{t \in \mc{V}} |Q_{v_0,H}(1+it)R_{v_0,H}(1+it)|^2.
$$
By assumption, for each $t \in \mc{V}$ we have $|Q_{r_0,H}(1+it)| > e^{-\alpha_J r_0/H_J} \geq P_J^{-\alpha_J}$, for some $r_0 \in \mc{I}_J$. We have $\log Q_{J+1} \geq \sqrt{\log X}$ by definition, and (as mentioned in Remark \ref{rem:params}) $\log P_J \geq \frac{4B}{\eta} \log\log Q_{J+1}$, whence $(\log X)^{2B/\eta} \leq P_J \leq Q_J \leq \exp(\sqrt{\log X})$. Applying Lemma \ref{lem:LSPrim}(b) for each $r_0 \in \mc{I}_J$, we thus have
$$
|\mc{V}| \ll_B |\mc{I}_J|\exp\left(2\alpha_J(\log P_J) \left(1+\frac{\log X}{\log P_J}\right) + 2B \frac{\log X \log\log X}{\log P_J}\right) \leq X^{1/2-2\eta + o(1)} \cdot X^{\eta} = X^{1/2-\eta + o(1)}.
$$  
We now split the set $\mc{V}$ into the subsets
\begin{align*}
\mc{V}_S &:= \{t \in \mc{V} : |Q_{v_0,H}(1+it)| \leq (\log X)^{-\frac{1}{2}B^2 - 10}\} \\
\mc{V}_L &:= \{t \in \mc{V} : |Q_{v_0,H}(1+it)| > (\log X)^{-\frac{1}{2}B^2 - 10}\};
\end{align*}
the exponent $B^2/2$ is present in order to cancel the $\log X$ power that arises from
$$
\mc{P}_{f^2}(X) \ll_B H(f;X) \mc{P}_f(X)^2 \ll_B (\log X)^{B^2} \mc{P}_f(X)^2.
$$
By a pointwise bound, Lemma \ref{lem:LSInts}(c) and the above estimate for $|\mc{V}| \geq |\mc{V}_S|$, we obtain
\begin{align*}
\sum_{t \in \mc{V}_S} |Q_{v_0,H}(1+it)R_{v_0,H}(1+it)|^2 &\ll (\log X)^{-B^2-19}(1+|\mc{V}_S|X^{-1/2}) \frac{e^{v_0/H}}{X}\sum_{X/(3e^{v_0/H}) < n \leq X/e^{v_0/H}} |f(n)|^2 \\
&\ll_{B,C} (\log X)^{-B^2-19} \mc{P}_{f^2}(X) \ll_B (\log X)^{-19} \mc{P}_f(X)^2.
\end{align*}
Consider next the contribution from $\mc{V}_L$. Applying Lemma \ref{lem:LSPrim}(b) once again, this time using the condition $|Q_{v_0,H}(1+it)| > (\log X)^{-\frac{1}{2}B^2-10}$, to obtain
$$
|\mc{V}_L| \ll_B \exp\left((B^2+20) \left(\log\log X\right) \left(1 + \frac{\log X}{\log P}\right) + 2B \frac{\log X \log\log X}{\log P}\right) = \exp\left((\log X)^{\kappa + o_B(1)}\right).
$$
Recall that $\Delta = (2B+5) \kappa \in (0,1)$. Applying Lemma \ref{cor:Hal} (with $Z = 1/\delta = (\log X)^{\Delta}$) together with Lemma \ref{lem:LSPrim}(a), noting that $\kappa < 1/3-\kappa$, we obtain
\begin{align*}
&\sum_{t \in \mc{V}_L} |Q_{v_0,H}(1+it)R_{v_0,H}(1+it)|^2 \\
&\ll_{A,B,C} \mc{P}_{f}(X)^2\left(\left(\frac{\log Q}{\log P}\right)^{B}\delta^{1/2} + \left(\frac{\log Q}{\log P}\right)^{3B}\frac{\log\log X}{(\log X)^{\sg}} \right)^2 \sum_{t \in \mc{V}_L} |R_{v_0,H}(1+it)|^2 \\
&\ll_B \frac{\mc{P}_{f}(X)^2}{(\log P)^2} \left((\log X)^{2B\kappa -\Delta} + (\log X)^{6B\kappa-2\sg+o(1)}\right) \left(1 + |\mc{V}_L|(\log X)^2 \exp\left(-(\log X)^{1/3-\kappa-o(1)}\right)\right)\\
&\ll_B \mc{P}_{f}(X)^2\left((\log X)^{2(B+1)\kappa - 2 -\Delta} + (\log X)^{(6B+2)\kappa-2-2\sg+o(1)}\right).
\end{align*}
Combining this estimate with the one for $\mc{V}_S$, then plugging this back into our estimate \eqref{eq:UInt}, we get
\begin{align*}
&\int_{\mc{U}}|F(1+it)|^2 dt\\
 &\ll_{A,B,C} (\log X)^{2+2\kappa} \left((\log X)^{-19} + (\log X)^{2(B+1)\kappa - 2 - \Delta} + (\log X)^{(6B+2)\kappa- 2-2\sg + o(1)}\right)\mc{P}_f(X)^2 \\
&+ \left(\frac{\log\log X}{(\log X)^{\kappa}}\right)^{\min\{1,A\}}\mc{P}_{f}(X)^2 \\
&\ll \left((\log X)^{-15} + (\log X)^{2(B+2)\kappa-\Delta} + (\log X)^{(6B+4)\kappa-2\sg} + \left(\frac{\log\log X}{(\log X)^{\kappa}}\right)^{\min\{1,A\}}\right)\mc{P}_{f}(X)^2 \\
&\ll \left(\frac{\log\log X}{(\log X)^{\kappa}}\right)^{\min\{1,A\}}\mc{P}_{f}(X)^2,
\end{align*}
since $2(B+2) \kappa - \Delta = -\kappa$, and $\sg \geq \hat{\sg} > (3B+5/2)\kappa$ by definition. This completes the proof of Proposition \ref{prop:L2DPInt}.

\begin{proof}[Proof of Theorem \ref{thm:MRDB}]
Let $f \in \mc{M}(X;A,B,C;\gamma,\sg)$. Set $h := h_1/(\log h_0)^A$, and select $P_1 = (\log h)^{40B/\eta}$ and $Q_1 = h$. We may assume that $X$ is larger than any constant depending on $B$, since otherwise Theorem \ref{thm:MRDB} follows with a sufficiently large implied constant; we may also assume $h$ is larger than any constant depending on $B$, since otherwise the theorem follows (again with a large enough implied constant depending at most on $B$) from Remark \ref{rem:trivBd} and Lemma \ref{lem:contS} (taking $P = Q = 3/2$, say). \\
By Lemma \ref{lem:contS}, we have
\begin{align*}
&\frac{2}{X}\int_{X/2}^{X} \left|\frac{1}{h_1} \sum_{x-h_1 < n \leq x} f(m) - \frac{1}{h_1}\int_{n-h_1}^n u^{it_0} du \cdot \frac{1}{h_2} \sum_{\substack{n-h_2 < m \leq n}} f(m)m^{-it_0}\right|^2 \\
&\ll_{A,B,C}  \frac{2}{X}\int_{X/2}^{X}\left|\frac{1}{h_1} \sum_{\substack{n-h_1 < m \leq n \\ m \in \mc{S}}} f(m) - \frac{1}{h_1}\int_{n-h_1}^n u^{it_0} du \cdot \frac{1}{h_2} \sum_{\substack{n-h_2 < m \leq n \\ m \in \mc{S}}} f(m)m^{-it_0}\right|^2\\
&+\left(\frac{\log\log h_0}{\log h_0}\right)^A \mc{P}_{f}(X)^2.
\end{align*}
Set $\delta := (\log X)^{-(2B+5)\kappa}$ once again, and note that $t_0(fn^{-it_0},X) = 0$ is admissible. Thus,
$$
\sum_{\ss{X/3 < n \leq X \\ n \in \mc{S}}} \frac{f(n)n^{-it_0}}{n^s} = F(s+it_0).
$$
By the Cauchy-Schwarz inequality, we obtain
\begin{align*}
&\frac{2}{X}\int_{X/2}^{X} \left|\frac{1}{h_1} \sum_{\substack{x-h_1 < m \leq x \\ m \in \mc{S}}} f(m) - \frac{1}{h_1}\int_{x-h_1}^x u^{it_0} du \cdot \frac{1}{h_2} \sum_{\substack{x-h_2 < m \leq x \\ m \in \mc{S}}} f(m)m^{-it_0}\right|^2 dx \\
&\ll \frac{1}{X} \int_{X/2}^{X} \left|\frac{1}{h_1} \sum_{\substack{x-h_1 < m \leq x \\ m \in \mc{S}}} f(m) - \frac{1}{2\pi h_1}\int_{I_{\delta}} F(1+it) \frac{x^{1+it} - (x-h_1)^{1+it}}{1+it}dt \right|^2 dx\\
&+ \frac{1}{X} \int_{X/2}^{X} \left|\frac{1}{h_1}\int_{x-h_1}^x u^{it_0} du\right|^2\left|\frac{1}{h_2}\sum_{\ss{x-h_2 < m \leq x \\ m \in \mc{S}}} f(m)m^{-it_0} - \frac{1}{2\pi h_2}\int_{-\delta^{-1}}^{\delta^{-1}} F(1+it+it_0) \frac{x^{1+it} - (x-h_2)^{1+it}}{1+it}dt \right|^2 dx\\
&+ \small \max_{X/2 < x \leq X} \left|\frac{1}{h_1} \int_{I_{\delta}} F(1+it) \frac{x^{1+it} - (x-h_1)^{1+it}}{1+it} dt - \frac{1}{h_1}\int_{x-h_1}^x u^{it_0} du \cdot \frac{1}{h_2} \int_{-\delta^{-1}}^{\delta^{-1}} F(1+i(t+t_0)) \frac{x^{1+it}-(x-h_2)^{1+it}}{1+it}dt \right|^2 \normalsize\\
& =: \mc{I}_1 + \mc{I}_2 + \mc{I}_3.
\end{align*}
Consider $\mc{I}_3$ first. Making a change of variables $w = t-t_0$ and using $u^{iw} = x^{iw} + O(\delta^{-1} h_1/X)$ for $|w| \leq \delta^{-1}$ and $u \in [x-h_1,x]$, we see that
\begin{align*}
&\frac{1}{h_1}\int_{I_{\delta}} F(1+it)\frac{x^{1+it}-(x-h_1)^{1+it}}{1+it} dt = \frac{1}{h_1}\int_{-\delta^{-1}}^{\delta^{-1}} F(1+i(t_0+w)) \int_{x-h_1}^x u^{it_0 + iw} du dw \\
&= \frac{1}{h_1} \int_{x-h_1}^x u^{it_0} du \cdot \int_{-\delta^{-1}}^{\delta^{-1}} F(1+i(t_0+w)) x^{iw} dw + O\left(\delta^{-2}\frac{h_1}{X}\max_{|w| \leq \delta^{-1}} |F(1+i(t_0+w))| \right).
\end{align*}
Similarly, we have
$$
\frac{1}{h_2} \int_{-\delta^{-1}}^{\delta^{-1}} F(1+i(t+t_0)) \frac{x^{1+it}-(x-h_2)^{1+it}}{1+it}dt = \int_{-\delta^{-1}}^{\delta^{-1}} F(1+i(t+t_0)) x^{it} dt + O\left(\delta^{-2} \frac{h_2}{X} \max_{|t| \leq \delta^{-1}} |F(1+i(t+t_0))|\right).
$$
Applying partial summation, dropping the condition $n \in \mc{S}$, and then using Lemma \ref{lem:Shiu}, we find
$$
\max_{|t| \leq \delta^{-1}} |F(1+i(t+t_0))| \ll \frac{1}{X} \sum_{X/3 < n \leq X} |f(n)| \ll_{B,C} \mc{P}_f(X),
$$
and thus as $h_1 \leq h_2 = X/(\log X)^{\hat{\sg}/2}$ we obtain
$$
\mc{I}_3 \ll_{B,C}  \left(\delta^{-2}\frac{h_2}{X}\right)^2 \mc{P}_f(X)^2 \ll (\log X)^{4(2B+5) \kappa - (8 B + 21) \kappa} \mc{P}_{f}(X)^2 = (\log X)^{-\kappa} \mc{P}_{f}(X).
$$
Next, we treat $\mc{I}_1$ and $\mc{I}_2$. Applying Proposition \ref{prop:Pars} to each of these integrals and trivially bounding the $u$ integral in $\mc{I}_2$, we find
\begin{align*}
\mc{I}_1 + \mc{I}_2 &\ll \int_{[-X/h_1,X/h_1] \bk I_{\delta}} |F(1+it)|^2dt + \max_{T \geq X/h_1} \frac{X}{h_1T} \int_T^{2T} |F(1+it)|^2 dt.
\end{align*}
By the argument in Remark \ref{rem:2ndInt} and our choice of $h$, the latter is bounded by
$$
\ll \int_{[-X/h,X/h] \bk I_{\delta}} |F(1+it)|^2 dt + (\log h_0)^{-A} \mc{P}_{f}(X)^2.
$$
But by Proposition \ref{prop:L2DPInt} this integral is bounded by
$$
\ll_{A,B,C} \left(\frac{(\log Q_1)^{1/3}}{P_1^{1/6-\eta}} + \left(\frac{\log\log X}{(\log X)^{\kappa}}\right)^{\min\{1,A\}}\right) \mc{P}_{f}(X)^2 \ll \left((\log h)^{-30B} + \left(\frac{\log\log X}{(\log X)^{\kappa}}\right)^{\min\{1,A\}}\right) \mc{P}_f(X)^2.
$$
Combining these steps, and using $A \leq B$ and $\log h \asymp \log h_1$, we obtain
\begin{align*}
&\frac{2}{X}\int_{X/2}^{X} \left|\frac{1}{h} \sum_{x-h_1 < n \leq x} f(m) - \frac{1}{h_1}\int_{n-h_1}^n u^{it_0} du \cdot \frac{1}{h_2} \sum_{n-h_2 < m \leq n} f(m)m^{-it_0}\right|^2 dx  \\
&\ll_{A,B,C} \left(\left(\frac{\log\log h_0}{\log h_0}\right)^A + \left(\frac{\log\log X}{(\log X)^{\kappa}}\right)^{\min\{1,A\}} \right)\mc{P}_{f}(X)^2,
\end{align*}
which completes the proof.
\end{proof}
\begin{proof}[Proof of Corollary \ref{cor:MRVers}]
It suffices to show that $\mc{M}(X;A,B,C) \subseteq \mc{M}(X;A,B,C;1,\sg)$ for any $0 < \sg < \sg_{A,B}$, where $\sg_{A,B}$ is given by \eqref{eq:sgAB}, after which point the result follows from Theorem \ref{thm:MRDB}. \\
Suppose $f\in \mc{M}(X;A,B,C)$. By definition, $f$ is $(A,X)$ non-vanishing in the sense of \cite{MRII}, and so satisfies the condition \eqref{eq:hyp3}  (or equivalently, $\gamma =1$ is admissible in \eqref{eq:hyp3'}). Observe furthermore that upon defining $\tilde{f}_B := f\cdot B^{-\Omega}$, $\tilde{f}_B$ takes values in $\mb{U}$ and satisfies 
$$
\sum_{z < p \leq w} \frac{|\tilde{f}_B(p)|}{p} \geq \frac{A}{B}\sum_{z < p \leq w} \frac{1}{p} - O_{A,B}\left(\frac{1}{\log z}\right) \text{ for all } 2 \leq z \leq w \leq X,
$$
with $A/B \in (0,1]$ necessarily. By \cite[Lem. 5.1(i)]{MRII}, we get that
\begin{align}
\rho(f,n^{it};X)^2 &= B\sum_{p \leq X} \frac{|\tilde{f}_B(p)| - \text{Re}(\tilde{f}_B(p)p^{-it})}{p} \nonumber\\
&\geq B \rho \min\{\log\log X, 3\log(|t-t_0| \log X+1)\} + O_{\rho,A}(1), \label{eq:lowBdRho}
\end{align}
for any $0 < \rho < \rho_{A/B}$ with $\rho_{\alpha}$ defined (see \cite[(14)]{MRII}) by 
$$
\rho_{\alpha} = \frac{\alpha}{3}\left(1-\text{sinc}(\pi \alpha/2)\right) > 0.
$$
Since $\sg_{A,B} = B \rho_{A/B}$, this implies that condition \eqref{eq:hyp4} holds with any $0 < \sg < \sg_{A,B}$. This completes the proof.
\end{proof}

\section{Applications}

\subsection{Proof of Theorem \ref{thm:momCusp}} \label{sec:RSProof}
Let $f$ be a primitive, Hecke-normalized holomorphic cusp form without complex multiplication of fixed even weight $k \geq 2$ and level 1. Write the Fourier expansion of $f$ at $\infty$ as
$$
f(z) = \sum_{n \geq 1} \lambda_f(n) n^{\frac{k-1}{2}} e(nz), \quad\quad \text{Im}(z) > 0,
$$
where $\{\lambda_f(n)\}_n$ is the sequence of normalized Fourier coefficients with $\lambda_f(1) = 1$. As noted in Section \ref{sec:RS} $\lambda_f$ is a multiplicative function. 
Deligne's proof of the Ramanujan conjecture for $f$ shows that $|\lambda_f(p)| \leq 2$ and $|\lambda_f(n)| \leq d(n)$ for all primes $p$ and positive integers $n$. Moreover, the quantiative Sato-Tate theorem of Thorner \cite{Tho}, which is based on the deep results of Newton and Thorne \cite{NeTh}, shows that for any $[a,b] \subseteq [-2,2]$,
\begin{equation}\label{eq:quantST}
|\{p \leq X: \lambda_f(p) \in [a,b]\}| =  \left(\frac{1}{\pi} \int_a^b \sqrt{1-(v/2)^2} dv\right) \int_2^X \frac{dt}{\log t} + O\left(\frac{X\log(k\log X)}{(\log X)^{3/2}}\right).
\end{equation}
Recall that
\[
c_\alpha = \frac{2^\alpha}{\sqrt{\pi}} \frac{\Gamma\left(\frac{\alpha + 1}{2}\right)}{\Gamma(\alpha/2+2)}.
\]
Using this data, we will prove the following.
\begin{prop}\label{prop:MABC}
Let $\alpha > 0$. There is a constant $\delta = \delta(\alpha) > 0$ such that $|\lambda_f|^{\alpha} \in \mc{M}(X;c_{\alpha},2^{\alpha},2;1/2-\e,\delta)$. 
\end{prop}
Using Proposition \ref{prop:MABC} we will be able to apply Theorem \ref{thm:MRFull} in order to derive Corollary \ref{cor:RankinSelberg}.\\
We will check that $|\lambda_f|^\alpha$ satisfies the required hypotheses in the following lemmas. 
\begin{lem} \label{lem:MertST}
For any $2 \leq z < w$ we have
$$
\sum_{z < p \leq w} \frac{|\lambda_f(p)|^\alpha}{p} = c_\alpha \sum_{z < p \leq w} \frac{1}{p} + O((\log z)^{-1/2+o(1)}).
$$
Similarly, we have
$$
\sum_{p \leq X} \frac{|\lambda_f(p)|^{2\alpha}}{p} = c_{2\alpha} \log\log X + O(1).
$$
\end{lem}
\begin{proof}
Let $\beta \in \{\alpha,2\alpha\}$. By partial summation,
$$
\sum_{z < p \leq w} \frac{|\lambda_f(p)|^\beta}{p} = \int_0^2 u^{\beta} d\{\sum_{\ss{z < p \leq w \\ |\lambda_f(p)| \leq u}} \frac{1}{p}\} = 2^\beta \sum_{z < p \leq w} \frac{1}{p} - \beta\int_0^2 \left(\sum_{\ss{z < p \leq w \\ |\lambda_f(p)| \leq u}} \frac{1}{p}\right) u^{\beta-1} du.
$$
Fix $u \in (0,2]$, and let $I_u := [0,u] \cup [-u,0]$ so that $|\lambda_f(p)| \leq u$ if and only if $\lambda_f(p) \in I_u$. By partial summation and \eqref{eq:quantST},
\begin{align*}
\sum_{\ss{z < p \leq w \\ \lambda_f(p) \in I_u}} \frac{1}{p} &= \frac{2}{\pi} \int_0^u \sqrt{1-(v/2)^2} dv \cdot \int_z^w \frac{dy}{y \log y} + O\left(\frac{\log\log z}{(\log z)^{3/2}} + \int_z^w \frac{\log\log y \, dy}{y(\log y)^{3/2}}\right) \\
&= \left(\frac{2}{\pi} \int_0^u \sqrt{1-(v/2)^2} dv\right) \log(\log w/\log z) + O((\log z)^{-1/2+o(1)}).
\end{align*}
Multiplying the main term by $\beta u^{\beta-1}$ and integrating in $u$, we obtain $I_\beta \log(\log w/\log z) + O(1/\log z)$, where 
\begin{align*}
I_\beta &:= \frac{2\beta}{\pi} \int_0^2 u^{\beta-1} \int_0^u \sqrt{1-(v/2)^2} dv du = \frac{2\beta}{\pi} \int_0^2 \left(\int_v^2 u^{\beta-1} du \right) \sqrt{1-(v/2)^2} dv  \\
&= \frac{2^{\beta+1}}{\pi}\int_0^2 (1-(v/2)^\beta) \sqrt{1-(v/2)^2} dv = 2^{\beta} - \frac{2^{\beta+1}}{\pi} \int_0^2 (v/2)^{\beta}\sqrt{1-(v/2)^2} dv.
\end{align*}
Making the change of variables $t := (v/2)^2$, we find that
\[
2^{\beta}-I_{\beta} = \frac{2^{\beta+1}}{\pi} \int_0^1 t^{(\beta-1)/2} (1-t)^{1/2} dt = \frac{2^{\beta+1}}{\pi} \frac{\Gamma\left(\frac{\beta + 1}{2}\right) \Gamma(3/2)}{\Gamma(\beta/2 + 2)} = \frac{2^{\beta}}{\sqrt{\pi}} \frac{\Gamma\left(\frac{\beta+1}{2}\right)}{\Gamma(\beta/2+2)} = c_{\beta}.
\]
It follows, therefore, that
\begin{align*}
\sum_{z < p \leq w} \frac{|\lambda_f(p)|^\alpha}{p} &= I_{\alpha} \sum_{z < p \leq w} \frac{1}{p} + O((\log z)^{-1/2+o(1)}) = c_{\alpha} \sum_{z < p \leq w} \frac{1}{p} + O((\log z)^{-1/2+o(1)}), \\
\sum_{z < p \leq w} \frac{|\lambda_f(p)|^{2\alpha}}{p} &= I_{2\alpha}\sum_{z < p \leq w} \frac{1}{p} + O((\log z)^{-1/2+o(1)}) = c_{2\alpha}\sum_{z < p \leq w} \frac{1}{p} + O((\log z)^{-1/2+o(1)}),
\end{align*}
and both claims follow.
\end{proof}
To obtain uniform lower bound estimates for $\rho(|\lambda_f|^{\alpha},n^{it};X)^2$ we will need some control over the product $|\lambda_f(p)|^\alpha (1-\cos(t\log p))$, on average over $p$. In some ranges of $t$ this is furnished by the following lemma.
\begin{lem} \label{lem:GoldLi}
Let $|t| \geq 1$. There is a constant $c = c(\alpha) > 0$ such that if $Y \geq (|t|+3)^2$ and $Y$ is sufficiently large then
$$
\sum_{Y < p \leq 2Y} |\lambda_f(p)|^\alpha|1-p^{it}|^2 \geq c\frac{Y}{\log Y}.
$$
\end{lem}
\begin{proof}
We adapt an argument due to Goldfeld and Li \cite[Lem. 12.12 and 12.15]{GoldLi} and Humphries \cite[Lem. 2.1]{Hump}. Let $\eta \in (0,1/2)$ be a parameter to be chosen later. By the prime number theorem for Rankin-Selberg $L$-functions \cite[Thm. 2]{Ran} we have
$$
\sum_{Y < p \leq 2Y} |\lambda_f(p)|^2 \log p = (1+o(1))Y.
$$
Since $|\lambda_f(p)|^2 \leq 4$ for all $p$, invoking the usual prime number theorem we obtain
\begin{align*}
(1+o(1))Y &\leq \eta^2 \sum_{\ss{Y < p \leq 2Y \\ |\lambda_f(p)| \leq \eta}} \log p + 4(\log 2Y) |\{Y < p \leq 2Y : |\lambda_f(p)| > \eta\}| \\
&\leq (\eta^2 + o(1)) Y + 4(\log 2Y) |\{Y < p \leq 2Y : |\lambda_f(p)| > \eta\}|,
\end{align*}
which rearranges as
\begin{equation}\label{eq:lrglambda}
|\{Y < p \leq 2Y : |\lambda_f(p)| > \eta\}| \geq \left(\frac{1-\eta^2}{4} - o(1)\right) \frac{Y}{\log Y}.
\end{equation}
Next, we estimate the cardinality
\begin{align*}
|\{Y < p \leq 2Y : |1-p^{it}| \leq \eta\}| &= |\{Y < p \leq 2Y : |\sin((t\log p)/2)| \leq \eta/2\}|. 
\end{align*}
Set $\beta := \sin^{-1}(\eta/2)/\pi \in [0,1/2]$, Whenever $\sin(t\log p/2) \in [-\eta/2, \eta/2]$ there is $m \in \mb{Z}$ such that $(t \log p)/2 \in [\pi (m - \beta), \pi (m + \beta)]$. By Jordan's inequality, $\beta \leq \frac{1}{2} \sin( \pi \beta) = \frac{\eta}{4}$, and we see that
\begin{align*}
|\{Y < p \leq 2Y : |1-p^{it}| \leq \eta\}| \leq |\{Y < p \leq 2Y : \|(t\log p)/(2\pi)\| \leq \eta/4\}|,
\end{align*}
where $\|t\| := \min_{m \in \mb{Z}} |t-m|$. Splitting up the primes $Y < p \leq 2Y$ according to the nearest integer $m$ to $(t\log p)/(2\pi)$, the latter may be bounded above as
$$
\leq \sum_{\frac{|t|\log Y+\eta/4}{2\pi} \leq m \leq \frac{|t|\log(2Y) - \eta/4}{2\pi}} \left(\pi\left(e^{\frac{2\pi m+\eta/4}{|t|}}\right) - \pi\left(e^{\frac{2\pi m-\eta/4}{|t|}}\right)\right).
$$
For each $m$ we have $e^{\frac{2\pi m+\eta/4}{|t|}} \leq \left(1+\frac{\eta}{2|t|}\right) e^{\frac{2\pi m-\eta/4}{|t|}}$.
By the Brun-Titchmarsh theorem we get, uniformly over all $m$ in the sum,
$$
\pi\left(e^{\frac{2\pi m+\eta/4}{|t|}}\right) - \pi\left(e^{\frac{2\pi m-\eta/4}{|t|}}\right) \leq \frac{\eta}{|t|} \frac{e^{2\pi(m-\eta/4)/|t|}}{2\pi(m-\eta/4)/|t|-\log(|t|/\eta)} \leq \eta \frac{2Y}{|t|(\log Y - (\log Y)/2)} \leq 8\eta\frac{Y}{|t|\log Y},
$$
for $Y$ large enough. Since there are $\leq 1+ |t|(\log(2Y)-\log Y)/(2\pi) \leq 2|t|$ integers in the range of summation, we obtain
\begin{equation}\label{eq:smallpit}
|\{Y < p \leq 2Y : |1-p^{it}| \leq \eta\}| \leq 16 \eta\frac{Y}{\log Y}.
\end{equation}
We deduce from \eqref{eq:lrglambda} and \eqref{eq:smallpit} that
\begin{align*}
|\{Y < p \leq 2Y : |\lambda_f(p)| > \eta \text{ and } |1-p^{it}| > \eta\}| &\geq |\{Y < p \leq 2Y : |\lambda_f(p)| > \eta\}| - |\{Y < p \leq 2Y : |1-p^{it}| \leq \eta\}| \\
&\geq \max\left\{0, \frac{1-64\eta-\eta^2}{4}-o(1)\right\} \cdot \frac{Y}{\log Y}.
\end{align*}
Selecting $\eta = \frac{1}{128}$, we get that if $Y$ is sufficiently large,
$$
|\{Y < p \leq 2Y : |\lambda_f(p)| > \eta \text{ and } |1-p^{it}| > \eta\}| \geq \frac{Y}{10 \log Y}.
$$
Finally, we obtain that
$$
\sum_{Y < p \leq 2Y} |\lambda_f(p)|^{\alpha} |1-p^{it}|^2 > \eta^{2+\alpha}|\{Y < p \leq 2Y : |\lambda_f(p)| > \eta \text{ and } |1-p^{it}| > \eta\}| \geq \frac{\eta^{2+\alpha}}{10} \cdot \frac{Y}{\log Y},
$$
which proves the claim with $c := 2^{-10(2+\alpha)}$.
\end{proof}


We now obtain lower bounds for $\rho(|\lambda_f|^\alpha,n^{it};X)^2$ for all $|t| \leq X$.
\begin{lem} \label{lem:check4}
There is a $\delta = \delta(\alpha) > 0$ such that whenever $|t| \leq X$ we have
$$
\sum_{p \leq X} \frac{|\lambda_f(p)|^\alpha(1- \cos(t\log p))}{p} \geq \delta \min\{\log\log X, \log(1+|t|\log X)\} - O(1).
$$
Moreover, we may select $t_0(|\lambda_f|^2,X) = 0$.
\end{lem}
\begin{proof}
We may assume that $X$ is sufficiently large, else the estimate given is trivial. When $|t| \leq 1/\log X$ the claim is vacuous. Thus, we may focus on the case $1/\log X < |t| \leq X$. We consider the ranges $1/\log X < |t| \leq 1$, $1 < |t| \leq \log X$ and $\log X < |t| \leq X$ separately. 
Throughout we will introduce an auxiliary parameter $2 \leq Y \leq X$, chosen case by case, and use the inequality
$$
\sum_{p \leq X} \frac{|\lambda_f(p)|^\alpha(1-\cos(t\log p)}{p} \geq \sum_{Y \leq p \leq X} \frac{|\lambda_f(p)|^\alpha (1-\cos(t\log p))}{p}.
$$
Suppose first that $1/\log X < |t| \leq 1$. Let $Y := e^{1/|t|}$ and write
$$
R(X) := \sum_{p \leq X} |\lambda_f(p)|^{\alpha} - c_{\alpha} \pi(X), 
$$
so that by arguments analogous to those of Lemma \ref{lem:MertST}, $R(X) \ll X/(\log X)^{3/2-o(1)}$. By partial summation and Lemma \ref{lem:MertST},
\begin{align*}
\sum_{Y \leq p \leq X} \frac{|\lambda_f(p)|^\alpha \cos(t\log p)}{p} &= c_{\alpha} \int_Y^X \frac{\cos(t\log u)}{u} \frac{du}{\log u} + O\left(\frac{1}{\log Y}\right) -  \int_Y^X R(u) \left(\cos(t\log u) - t\sin(t\log u)\right)\frac{du}{u^2} \\
&= c_{\alpha}\int_{1}^{|t|\log X} \cos( tv/|t| ) \frac{dv}{v} + O\left((1+|t|)(\log Y)^{-1/2+o(1)}\right) \ll 1,
\end{align*}
the bound in the last step arising from setting $v := |t| \log u$ and integrating by parts. Thus, in light of Lemma \ref{lem:MertST}
$$
\sum_{Y < p \leq X} \frac{|\lambda_f(p)|^\alpha (1-\cos(t\log p))}{p} = c_{\alpha}\log(1+ |t|\log X) - O(1).
$$
Next, we consider the intermediate range $1 < |t| \leq \log X$. Here, we set $Y := (10\log X)^2$, employ a dyadic decomposition and apply Lemma \ref{lem:GoldLi} to obtain, when $X$ is large enough,
\begin{align*}
\sum_{p \leq X} \frac{|\lambda_f(p)|^\alpha(1-\cos(t\log p))}{p} &\geq \frac{1}{2} \sum_{Y < 2^j \leq X/2} 2^{-(j+1)} \sum_{2^j <p \leq 2^{j+1}} |\lambda_f(p)|^\alpha|1-p^{it}|^2 \geq \frac{c}{4\log 2} \sum_{Y < 2^j \leq X/2} \frac{1}{j} \\
&= \frac{c}{4\log 2} \log\left(\frac{\log X}{\log Y}\right) - O(1) \geq \frac{c}{8\log 2} \log\log X - O(1),
\end{align*}
for some $c = c(\alpha) > 0$. \\
Finally, assume that $\log X \leq |t| \leq X$. In this case, put $Y  := \exp\left((\log X)^{2/3+\e}\right)$, where $\e > 0$ is small. Let $m \geq 1$ be an integer parameter to be chosen later. By H\"{o}lder's inequality,
\begin{align*}
\left|\sum_{Y < p \leq X} \frac{|\lambda_f(p)|^\alpha\cos(t\log p)}{p}\right| \leq \left(\sum_{Y < p \leq X} \frac{|\lambda_f(p)|^{\frac{2\alpha m}{2m-1}}}{p}\right)^{1-1/2m} \cdot \left(\sum_{Y < p \leq X} \frac{\cos(t\log p)^{2m}}{p}\right)^{1/2m}.
\end{align*}
We can bound $|\lambda_f(p)|^{2\alpha m/(2m-1)} \leq |\lambda_f(p)|^\alpha 2^{\alpha/(2m-1)}$, so that by Lemma \ref{lem:MertST} the first sum is
$$
\leq 2^{\alpha/(2m)} c_{\alpha}^{1-1/(2m)} (\log(\log X/\log Y))^{1-1/2m} + O_m(1).
$$
Now, we can write
$$
\cos(t\log p)^{2m} = 2^{-2m} (p^{it} + p^{-it})^{2m} = 2^{-2m} \binom{2m}{m} + 2^{-2m} \sum_{0 \leq j \leq m-1} \binom{2m}{j} (p^{2i(m-j)t} +p^{-2i(m-j)t}).
$$
By the zero-free region of the Riemann zeta function (see e.g., \cite[Lem. 2]{MRLiou}) we have
$$
\max_{1 \leq |l| \leq m} \left|\sum_{Y < p \leq X} \frac{1}{p^{1+ilt}}\right| \ll \frac{\log X}{1+|t|} + (\log X)^{-10} \ll 1.
$$
It follows that
$$
\sum_{Y < p \leq X} \frac{\cos(t \log p)^{2m}}{p} = 2^{-2m} \binom{2m}{m}\log(\log X/\log Y) + O_m(1),
$$
and therefore in sum we have
$$
\left|\sum_{Y < p \leq X} \frac{|\lambda_f(p)|^\alpha \cos(t\log p)}{p}\right| \leq c_{\alpha} \left(\frac{2^{\alpha}}{c_{\alpha}}\right)^{1/(2m)} \left(2^{-2m} \binom{2m}{m}\right)^{1/2m} \log\left(\frac{\log X}{\log Y}\right) + O_m(1).
$$
Using the bounds $\sqrt{2\pi n} (n/e)^n \leq n! \leq 2\sqrt{2\pi n} (n/e)^n$, valid for all $n \in \mb{N}$ (see e.g., \cite{Rob}), we get that 
$$
\left(\frac{2^{\alpha}}{c_{\alpha}}\right)^{1/(2m)} \left(2^{-2m} \binom{2m}{m}\right)^{1/(2m)} \leq \left(\frac{2^{\alpha+1}}{c_{\alpha}\sqrt{\pi m}}\right)^{1/(2m)},
$$
and thus taking $m \geq m_0(\alpha)$, we obtain
$$
\left|\sum_{Y < p \leq X} \frac{|\lambda_f(p)|^\alpha \cos(t\log p)}{p}\right| \leq 2^{-1/(2m)} c_{\alpha} \log(\log X/\log Y) + O(1),
$$
for $\eta \in (0,1)$. We thus obtain in this case that
$$
\sum_{p \leq X} \frac{|\lambda_f(p)|^\alpha (1-\cos(t\log p))}{p} \geq c_\alpha (1-2^{-1/(2m)}) \left(\frac{1}{3}-\e\right) \log\log X - O(1).
$$
Combining our estimates from each of these ranges and putting $\delta := c_\alpha \min\{c/(8\log 2),  \frac{1}{4}(1-2^{-1/(2m)})\}$, we deduce that 
$$
\sum_{p \leq X} \frac{|\lambda_f(p)|^\alpha(1-\cos(t\log p))}{p} \geq \delta \min\{\log\log X , \log(1+|t|\log X)\} - O(1).
$$
This completes the proof of the lower bound. \\
Finally, note that $\rho(|\lambda_f|^\alpha,1;X) = 0$, so $t_0 = 0$ is a minimizer of $t\mapsto \rho(|\lambda_f|^\alpha,n^{it};X)^2$. This completes the proof.
\end{proof}

\begin{proof}[Proof of Proposition \ref{prop:MABC}]
Let $X$ be large. By Deligne's theorem we have $|\lambda_f(p)|^\alpha \leq 2^\alpha$ for all $p$. In addition, we have $|\lambda_f(n)|^\alpha \leq d(n)^\alpha \leq d_{2^{\alpha+1}}(n)^{\max\{1,\alpha\}}$ for all $n$. By Lemma \ref{lem:MertST} we see that we can take $A = c_\alpha$ and any $\gamma \in (0,1/2)$ in an estimate of the form
$$
\sum_{z < p \leq w} \frac{|\lambda_f(p)|^\alpha}{p} \geq c_\alpha\sum_{z < p \leq w} \frac{1}{p} - O\left(\frac{1}{(\log z)^{\gamma}}\right) \text{ for } 2 \leq z \leq w \leq X,
$$
and finally by Lemma \ref{lem:check4} there is a constant $\delta > 0$ such that 
$$
\sum_{p \leq X} \frac{|\lambda_f(p)|^\alpha-\text{Re}(|\lambda_f(p)|^\alpha p^{-it})}{p} \geq \delta \min\{ \log\log X, \log(1+|t-t_0| \log X)\} - O(1),
$$
as $t_0 =0$ is admissible. Thus, by definition, $|\lambda_f|^\alpha \in \mc{M}(X; c_\alpha, 2^\alpha, \max\{1,\alpha\}; 1/2-\e,\delta)$ for any $\e \in (0,1/2)$ and $X$ large, as claimed. 
\end{proof}

\begin{proof}[Proof of Theorem \ref{thm:momCusp}]
In view of Proposition \ref{prop:MABC} we may apply Theorem \ref{thm:MRFull} to the function $|\lambda_f|^\alpha$. We see that if $h = h_0 H(|\lambda_f|^\alpha;X)$ and $h_0 \ra \infty$ then
\begin{align*}
&\frac{1}{X}\int_X^{2X} \left|\frac{1}{h} \sum_{x < n \leq x+h} |\lambda_f(n)|^\alpha - \frac{1}{X}\sum_{X < n \leq 2X} |\lambda_f(n)|^\alpha \right|^2 dx \\
&\ll \left( \left(\frac{\log\log h_0}{\log h_0}\right)^{c_{\alpha}} + \frac{\log\log X}{(\log X)^{\theta}}\right) (\log X)^{2(c_{\alpha}-1)},
\end{align*}
%
where $\theta = \theta(\alpha) > 0$. Furthermore, by Lemma \ref{lem:MertST}, we have
$$
H(|\lambda_f|^\alpha;X) \asymp \exp\left(\sum_{p \leq X} \frac{|\lambda_f(p)|^{2\alpha} - 2|\lambda_f(p)|^\alpha + 1}{p}\right) \asymp d_\alpha \log X,
$$
where $d_\alpha = c_{2\alpha} - 2c_\alpha + 1$. 
Thus, changing $h_0$ by a constant factor, we deduce that $h = h_0 (\log X)^{d_\alpha}$ can be taken in the above estimate. The claim follows.  
\end{proof}

\begin{proof}[Proof of Corollary \ref{cor:RankinSelberg}]
When $\alpha = 2$ we have $c_2 = 1$ and $d_2 = c_4 - 2c_2 + 1 = 1$. Hence, $H(|\lambda_f|^2;X) \asymp \log X$. Since 
$$
c_f = \frac{1}{X}\sum_{X  <n \leq 2X} |\lambda_f(n)|^2 + O(X^{-2/5}),
$$
the claim for $|\lambda_f|^2$ follows immediately from Theorem \ref{thm:momCusp}.\\
It remains to consider $g_f(n) = \sum_{d^2|n} |\lambda_f(n/d^2)|^2$. Clearly, $g_f(n) \leq \sum_{e|n} d(n/e)^2 \leq d(n)^3$, and $g_f(p) = |\lambda_f(p)|^2$ for all primes $p$. By Proposition \ref{prop:MABC} we have $g_f \in \mc{M}(X; 1,4,3;1/2-\e,\delta)$ for any $\e \in (0,1/2)$, and $H(g_f;X) = H(|\lambda_f|^2;X)$. The result now follows in the same was as for $|\lambda_f|^2$, using 
$$
d_f = \frac{1}{X}\sum_{X < n \leq 2X} g_f(n) + O(X^{-2/5})
$$
in this case.
\end{proof}

\subsection{Proof of Theorem \ref{thm:genAutForms}}
Fix $m \geq 2$, let $\mb{A}$ denote the adeles over $\mb{Q}$, and let $\pi$ be a fixed cuspidal automorphic representation for $GL_m(\mb{A})$ with unitary central character normalized so that it is trivial on the diagonally embedded copy of $\mb{R}^+$. We write $q_{\pi}$ to denote the conductor of $\pi$. We assume that $\pi$ satisfies GRC, and we write $\lambda_{\pi}(n)$ to denote the $n$th coefficient of the standard $L$-function of $\pi$. 
The key result of this subsection is the following analogue of Proposition \ref{prop:MABC}.
\begin{prop}\label{prop:MABCGen}
With the above notation, we have $|\lambda_{\pi}|^2 \in \mc{M}(X; 1, m^2, 2)$ whenever $X$ is sufficiently large in terms of $q(\pi)$, and $t_0(|\lambda_{\pi}|^2,X) = 0$ is admissible.
\end{prop}
For primes $p \nmid q_{\pi}$ we have $|\lambda_{\pi}(p)|^2 = \lambda_{\pi \otimes \tilde{\pi}}(p)$, where $\tilde{\pi}$ denotes the contragredient representation of $\pi$ and $\pi \otimes \tilde{\pi}$ is the Rankin-Selberg convolution of $\pi$ and $\tilde{\pi}$. As $\pi$ is fixed, the primes $p|q_{\pi}$ will cause no harm to our estimates.
\begin{lem}\label{lem:PNTRS}
There is a constant $c = c(m)  > 0$ such that 
$$
\sum_{p \leq X} |\lambda_\pi(p)|^2 \log p = X + O_{\pi}(X e^{-c\sqrt{\log X}}) \text{ as $X \ra \infty$.}
$$
In particular, we have $\sum_{z < p \leq w} \frac{|\lambda_{\pi}(p)|^2}{p} = \sum_{z < p \leq w} \frac{1}{p} + O_{\pi}(\frac{1}{\log z})$ for any $z < p \leq w$.
\end{lem}
\begin{proof}
Assume $X$ is sufficiently large relative to $q_{\pi}$. Set $f := \pi \otimes \tilde{\pi}$, and write $\Lambda_f(n)$ to denote the $n$th coefficient of the logarithmic derivative $-\frac{L'}{L}(s,f)$. Combining \cite[Thm. 5.13]{IK} (see the remarks that follow the statement for a discussion relevant to the case of a Rankin-Selberg convolution) with \cite[Exer. 6]{IK}, we deduce that
$$
\sum_{p \leq X} |\lambda_{\pi}(p)|^2 \log p = \sum_{n \leq X} \Lambda_f(n) + O_{\pi}(\sqrt{X} \log^2 X) = X + O_{\pi}\left(X \exp\left(-c'm^{-4} \sqrt{\log X}\right)\right),
$$
for some absolute constant $c' > 0$ (note that the exceptional zero plays no role when $X$ is large enough). This implies the first claim.
The second follows immediately from the first statement by partial summation.
\end{proof}
\begin{proof}[Proof of Proposition \ref{prop:MABCGen}]
Lemma \ref{lem:PNTRS} implies that \eqref{eq:hyp3} holds with $A = 1$, and by GRC we have $|\lambda_{\pi}(n)|^2 \leq d_m(n)^2$. It thus follows that $|\lambda_{\pi}|^2 \in \mc{M}(X;1,m^2,2)$, as claimed. That the choice $t_0(|\lambda_{\pi}|^2,X) = 0$ is admissible is obvious, as in the proof of Proposition \ref{prop:MABC}.
\end{proof}
\begin{proof}[Proof of Theorem \ref{thm:genAutForms}]
Recall that $m$ and $\pi$ are fixed. A direct application of Corollary \ref{cor:MRVers} shows (after replacing $X$ by $2X$ an $x-h$ by $x$) that there is $\kappa = \kappa(m) > 0$ such that 
\begin{align*}
&\frac{1}{X}\int_{X}^{2X} \left|\frac{1}{h}\sum_{x < n \leq x+h} |\lambda_{\pi}(n)|^2 - \frac{1}{X}\sum_{X < n \leq 2X} |\lambda_{\pi}(n)|^2\right|^2 dx \ll_m \left(\frac{\log\log h_0}{\log h_0} + \frac{\log\log X}{(\log X)^{\kappa}}\right) \mc{P}_{|\lambda_{\pi}|^2}(X)^2,
\end{align*}
where $h = h_0 H(|\lambda_{\pi}|^2;X)$ and $10 \leq h_0 \leq X/(10 H(|\lambda_{\pi}|^2;X))$. 
By Lemma \ref{lem:PNTRS} $\mc{P}_{|\lambda_{\pi}|^2}(X) \ll_m 1$. We also have
$$
H(|\lambda_{\pi}|^2;X) \asymp_m \exp\left(\sum_{p \leq X} \frac{|\lambda_{\pi}(p)|^4 - 2|\lambda_{\pi}(p)|^2 + 1}{p}\right) \ll \exp\left(\sum_{p \leq X} \frac{m^2 |\lambda_{\pi}(p)|^2 - 1}{p}\right) \ll (\log X)^{m^2-1}.
$$
It follows that our variance estimate holds if $h \geq h_0 (\log X)^{m^2-1}$ and $10 \leq h_0 \leq X/(10(\log X)^{m^2-1})$, and the proof of the theorem is complete.
\end{proof}

\subsection{Proof of Corollary \ref{cor:Hooley}}
Recall that
\[
\Delta(n) := \max_{u \in \mb{R}} \sum_{ \ss{ d|n \\ e^u \leq d < e^{u+1}}} 1.
\]
Given $\theta \in \mb{R}$, write also
$$
d(n,\theta) := \sum_{d|n} d^{i\theta}.
$$
This is clearly a multiplicative function, with $d(n) = d(n,0)$. In probabilistic terms, it is also $d(n)$ times the characteristic function of the distribution function 
\[
\mc{D}_n(v) := \frac{1}{d(n)} \sum_{ \ss{ d|n \\ d \leq e^v} } 1, \quad v \in \mb{R}.
\]
introduced in Section \ref{subsec:apps}. 
The following general bounds for concentration functions in terms of characteristic functions allow us to relate $\Delta(n)$ with integral averages of $d(n,\theta)$.
\begin{lem} \label{lem:TenExer}
There are constants $c_2 > c_1 > 0$ such that, uniformly in $n \in \mb{N}$,
$$
c_1\frac{1}{d(n)} \int_0^1 |d(n,\theta)|^2 d\theta \leq \Delta(n) \leq c_2\int_0^1 |d(n,\theta)| d\theta.
$$
\end{lem}
\begin{proof}
This is a special case of \cite[Lem. 30.2]{HallTenBook}.
\end{proof}

By Lemma \ref{lem:TenExer} we find that for any $x \in [X/2,X]$ and $10 \leq h \leq X$,
\begin{equation} \label{eq:lowBdDelta}
\sum_{x-h < n \leq x} \Delta(n) \gg \int_0^1 \sum_{x-h < n \leq x} \frac{|d(n,\theta)|^2}{d(n)} d\theta.
\end{equation}
For $\theta \in \mb{R}$ and $n \in \mb{N}$ we write $f_{\theta}(n) := |d(n,\theta)|^2/d(n)$. 
\begin{cor} \label{cor:fthetaApp}
Let $\theta \in (1/\log X,1]$. Let $10 \leq h_0 \leq X/\log X$ and put $h := h_0 (\theta^{-1} \log X)^{1/2}$. Then there is a constant $\kappa_{1,2} > 0$ such that for any $0 < \kappa < \kappa_{1,2}$,
$$
\frac{2}{X}\int_{X/2}^X \left|\frac{1}{h}\sum_{x-h < n \leq x} f_{\theta}(n) - \frac{2}{X} \sum_{X/2 < n \leq X} f_{\theta}(n)\right|^2 dx \ll \left(\frac{\log\log h_0}{\log h_0} + \frac{\log\log X}{(\log X)^{\kappa}}\right),
$$
The implicit constant is independent of $\theta$.
\end{cor}
To this end, we prove that $f_{\theta} \in \mc{M}(X;1,2,1)$ for all $1/\log X < \theta \leq 1$, which is the purpose of the following lemmas.
\begin{lem}\label{lem:primeftheta}
Let $\theta \in (0, 1]$, and let $\beta := \min\{1/\theta,\log X\}$. Then
$$
\sum_{p \leq X} \frac{f_{\theta}(p)}{p} = \log(\beta \log X) + O(1).
$$
Similarly,
$$
H(f_{\theta};X) \asymp (\beta \log X)^{1/2}.
$$
\end{lem}
\begin{proof}
Observe that for each $p$,
$$
f_{\theta}(p) = \frac{1}{2}|1+p^{i\theta}|^2 = 1+\cos(\theta \log p).
$$
Put $Y := \min\{X,\exp(1/\theta)\} = e^{\beta}$. For $p \leq Y$ we have $\cos(\theta \log p) = 1 +O(\theta^2(\log p)^2)$, so that by the prime number theorem,
$$
\sum_{p \leq Y} \frac{f_{\theta}(p)}{p} = \sum_{p \leq Y} \frac{2}{p} + O\left(\theta^2 \sum_{p \leq e^{1/\theta}} \frac{(\log p)^2}{p}\right) = 2 \log(\min\{1/\theta,\log X\}) + O(1).
$$
This proves the first claim if $0 \leq \theta \leq 1/\log X$, so assume now that $1/\log X < \theta \leq 1$. By partial summation and the prime number theorem, we have
\begin{align*}
\sum_{Y < p \leq X} \frac{1+\cos(\theta \log p)}{p} = \int_1^{\theta \log X} (1+\cos v) \frac{dv}{v} + O(1) &= \left(\frac{1}{2\pi} \int_0^{2\pi} (1+\cos u) du\right) \log(\theta \log X) + O(1) \\
&= \log(\theta \log X) + O(1).
\end{align*}
We thus deduce that
$$
\sum_{p \leq X} \frac{f_{\theta}(p)}{p} = 2\log(1/\theta) + \log(\theta \log X) = \log(\theta^{-1} \log X),
$$
and the first claim follows for all $\frac{1}{\log X} < \theta \leq 1$ as well. \\
For the second claim, we simply note that
$$
H(f_{\theta};X) \asymp \exp\left(\sum_{p \leq X} \frac{(f_{\theta}(p)-1)^2}{p}\right) = \exp\left(\sum_{p \leq X} \frac{\cos(\theta \log p)^2}{p}\right).
$$
A similar partial summation argument shows that
\begin{align*}
\sum_{p \leq X} \frac{\cos(\theta \log p)^2}{p}  &= \log(\min\{1/\theta,\log X\}) + \left(\frac{1}{2\pi} \int_0^{2\pi} (\cos u)^2 du\right) \log(1+\theta \log X) + O(1) \\
&= \frac{1}{2}\log(\min\{\log X,1/\theta\}\log X) + O(1),
\end{align*}
and the claim follows.
\end{proof}

\begin{lem} \label{lem:check3}
Let $\theta \in (0,1]$. Let $2 \leq z \leq w \leq X$. Then
$$
\sum_{z <  p \leq w} \frac{f_{\theta}(p)}{p} \geq \sum_{z < p \leq w} \frac{1}{p} + O(1/\log z).
$$
\end{lem}
\begin{proof}
Set $Y := \min\{X,e^{1/\theta}\}$ once again. If $w \leq Y$ then $\cos(\theta\log p) \geq \cos(1) \geq 0$ for all $z < p \leq w$, and thus 
$$
\sum_{z < p \leq w} \frac{f_{\theta}(p)}{p} \geq \sum_{z < p \leq w} \frac{1}{p} + O(1/\log z).
$$
On the other hand, if $Y \leq z$ (so that $\theta > 1/\log X$) then by the same partial summation argument as in Lemma \ref{lem:primeftheta} we find that
$$
\sum_{z < p \leq w} \frac{f_{\theta}(p)}{p} = \left(\frac{1}{2\pi} \int_0^{2\pi} (1+\cos u) du\right) \log(\log w/\log z) + O(1/\log z) = \sum_{z < p \leq w} \frac{1}{p} + O(1/\log z).
$$
Finally, suppose $z < Y < w$. In this case, we split the interval into the segments $(z,Y]$ and $(Y,w]$ and apply the arguments in each of the previous two cases to obtain
$$
\sum_{z < p \leq w} \frac{f_{\theta}(p)}{p} \geq \sum_{z < p \leq Y} \frac{1}{p} + \sum_{Y < p \leq w} \frac{1}{p} + O(1/\log z) \geq \sum_{z < p \leq w} \frac{1}{p} + O(1/\log z),
$$
as claimed.
\end{proof}

\begin{proof}[Proof of Corollary \ref{cor:fthetaApp}]
Let $h = h_0 (\theta^{-1} \log X)$. Note that $|d(n;\theta)|^2/d(n) \leq d(n)$ uniformly over $n$, so that combined with Lemma \ref{lem:check3} we have that $f_{\theta} \in \mc{M}(X;1,2,1)$ for all $\theta \in (1/\log X,1]$. As in Lemma \ref{lem:check4}, $t_0(f_{\theta}, X) = 0$ is admissible for all $\theta \in (1/\log X,1]$. Since $H(f_{\theta};X) \ll (\theta^{-1} \log X)^{1/2}$ uniformly over all $\theta \in (1/\log X,1]$ by Lemma \ref{lem:primeftheta}, the claim follows from Corollary \ref{cor:MRVers}.
\end{proof}

\begin{proof}[Proof of Corollary \ref{cor:Hooley}]
Let $\delta \in (0,1]$, set $Y := \exp\left((\log X)^{\delta}\right)$ and put $h = h_0 (\log X)^{(1+\delta)/2}$, with $10 \leq h_0 \leq X/10(\log X)^{(1+\delta)/2}$. By Lemma \ref{lem:primeftheta} we have $H(f_{\theta};X) \gg (\log X)^{(1+\delta)/2}$ for all $1/\log Y < \theta \leq 1$. By Fubini's theorem, the Cauchy-Schwarz inequality and Corollary \ref{cor:fthetaApp}, we thus see that
\begin{align*}
&\frac{2}{X}\int_{2/X}^X \left[\int_{1/\log Y}^1 \left|\frac{1}{h}\sum_{x-h < n \leq x} f_{\theta}(n) - \frac{2}{X}\sum_{X/2 < n \leq X} f_{\theta}(n)\right| d\theta \right] dx \\
&\leq \int_{1/\log Y}^1 \left(\frac{2}{X}\int_{X/2}^X \left|\frac{1}{h}\sum_{x- h < n \leq x} f_{\theta}(n)-\frac{2}{X}\sum_{X/2 < n \leq X} f_{\theta}(n)\right|^2 dx\right)^{1/2} d\theta \\
&\ll \left(\sqrt{\frac{\log\log h_0}{\log h_0}} + \left(\frac{\log\log X}{(\log X)^{\kappa}}\right)^{\frac{1}{2}} \right) \int_{1/\log Y}^1 \mc{P}_{f_{\theta}}(X) d\theta,
\end{align*}
for any $0 < \kappa < \kappa_{1,2}$. By Lemma \ref{lem:primeftheta} we have
$$
\int_{1/\log Y}^1 \mc{P}_{f_{\theta}}(X) d\theta \ll \int_{1/\log Y}^1 \exp\left(\sum_{p \leq X} \frac{f_{\theta}(p)-1}{p}\right) d\theta \ll \int_{1/\log Y}^1 \frac{d\theta}{\theta} = \log\log Y.
$$
We thus deduce that for all but $o_{h_0 \ra \infty}(X)$ exceptional integers $x \in [X/2,X]$, we have that
$$
\int_{1/\log Y}^1 \left|\frac{1}{h}\sum_{x-h < n \leq x} f_{\theta}(n) - \frac{2}{X}\sum_{X/2 < n \leq X} f_{\theta}(n)\right| d\theta = o_{h_0 \ra \infty}(\log\log Y).
$$
For any of the non-exceptional $x$, we apply \eqref{eq:lowBdDelta} to give
\begin{align}
&\frac{1}{h} \sum_{x-h < n \leq x} \Delta(n) \geq \int_0^1 \left(\frac{1}{h} \sum_{x-h < n \leq x} f_{\theta}(n)\right) d\theta \nonumber \\
&\geq \int_{1/\log Y}^1 \left(\frac{2}{X}\sum_{X/2 < n \leq X} f_{\theta}(n)\right) d\theta - \int_{1/\log Y}^1 \left|\frac{1}{h} \sum_{x-h < n \leq x} f_{\theta}(n) - \frac{2}{X}\sum_{X/2 < n \leq X}f_{\theta}(n)\right| d\theta \nonumber\\
&= \int_{1/\log Y}^1 \left(\frac{2}{X}\sum_{X/2 < n \leq X} f_{\theta}(n)\right) d\theta - o_{h_0 \ra \infty}(\log\log Y). \label{eq:lowBdHooley}
\end{align}
On the other hand, by \cite[Exer. 208]{Ten} we find that when $1/\log X \leq \theta \leq 1$,
$$
\frac{2}{X}\sum_{X/2 < n \leq X} \frac{|d(n,\theta)|^2}{d(n)} \geq \frac{2}{X}\sum_{X/2 < n \leq X} \frac{\mu^2(n)|d(n,\theta)|^2}{d(n)} = |\zeta(1+i\theta)| H_{\theta}(1) + O(|\theta^{3/2}|/\sqrt{\log X}),
$$
where for $\text{Re}(s) > 3/4$, $H_{\theta}(s)$ is some convergent Dirichlet series satisfying $H_{\theta}(1) \gg 1$ uniformly in $\theta \in [1/\log X,1]$. Integrating over $\theta \in [1/\log Y,1]$ and using the Laurent expansion $\zeta(1+i\theta) = (i\theta)^{-1} + O(1)$, we deduce that
\begin{equation}\label{eq:longHooSum}
\int_{1/\log Y}^1 \left(\frac{2}{X}\sum_{X/2 < n \leq X} f_{\theta}(n)\right) d\theta \gg \int_{1/\log Y}^1 \frac{d\theta}{\theta} + O(1) = \delta \log\log X + O(1).
\end{equation}
We thus have obtained
$$
\frac{1}{h}\sum_{x-h < n \leq x} \Delta(n) \gg \delta \log\log X
$$
for all but $o_{h_0 \ra \infty}(X)$ integers $x \in [X/2,X]$, and the claim follows.
\end{proof}

\section*{Acknowledgments}
The author warmly thanks 
Oleksiy Klurman and Aled Walker 
for helpful suggestions about improving the exposition of the paper, as well as for their encouragement. He is also grateful to Maksym Radziwi\l\l{} and Jesse Thorner for helpful conversations and suggestions regarding the applications to automorphic forms. Finally, he would like to thank G\'{e}rald Tenenbaum for useful comments and references. Much of this paper was written while the author held a Junior Fellowship at the Mittag-Leffler institute for mathematical research during the Winter of 2021. He would like to thank the institute for its support.

\bibliography{ErdConj}
\bibliographystyle{plain}
\end{document}